  \documentclass[reqno, 11 pt]{amsart}
      
\usepackage{amssymb}       
\usepackage{amsmath}        
\usepackage{euscript}      
\usepackage{amsthm}        
\usepackage{dsfont}     
\usepackage{amsfonts}       
\usepackage{marginnote}    
\usepackage{latexsym} 
\usepackage{amsopn} 
\usepackage{geometry} 
\usepackage{amscd} 
\usepackage{mathrsfs}  
\usepackage{caption}
\usepackage{aurical}   
\usepackage[english]{babel}
\usepackage[utf8]{inputenc}
\usepackage{graphicx}
\usepackage[dvipsnames]{xcolor}
\usepackage{multicol}
\usepackage{nomencl}
\makeglossary
\makenomenclature  
\usepackage{refcount}
\usepackage{upgreek}
\usepackage{float}
\usepackage{hyperref}
\usepackage{times}
\usepackage[titletoc]{appendix}

\newcommand{\comment}[1]{}

\newcommand{\e}{{\varepsilon}}
\newcommand{\vphi}{{\varphi}}

\newtheorem{thm}{Theorem}[section]
\newtheorem*{thm*}{Theorem}
\newtheorem{prop}[thm]{Proposition}
\newtheorem{lemma}[thm]{Lemma}

\newtheorem*{cor*}{Corollary}
\newtheorem{rmk}[thm]{Remark}
\newtheorem{ex}{Example}

\numberwithin{equation}{section}

\newcommand{\ii}{{\rm i}}

\newcommand{\C}{{\mathbb C}}

\newcommand{\R}{{\mathbb R}}

\newcommand{\T}{{\mathbb T}}
\newcommand{\Z}{{\mathbb Z}}

\usepackage{bm}

\newcommand{\nnorm}[1]{{\left\vert\kern-0.25ex\left\vert\kern-0.25ex\left\vert #1 
    \right\vert\kern-0.25ex\right\vert\kern-0.25ex\right\vert}}

\usepackage{framed,enumitem}

\makeatletter

\renewcommand{\tocsection}[3]{%
\indentlabel{\@ifnotempty{#2}{\bfseries\ignorespaces#1 #2\quad}}\bfseries#3}
\def\l@subsection{\@tocline{2}{0pt}{2.5pc}{5pc}{}}
\def\l@subsubsection{\@tocline{3}{0pt}{4.5pc}{5pc}{}}
\renewcommand\tocchapter[3]{%
  \indentlabel{\@ifnotempty{#2}{\ignorespaces#2.\quad}}#3%
}

\theoremstyle{remark}

\begin{document} 
 
\title[2D Kuramoto-Sivashinsky]{Nonlinear Instability in the 2D Kuramoto-Sivashinsky equation}


%

\date{}

\author[D. Ambrose {\em et Al.}]{David M. Ambrose}
\address{\scriptsize{Department of Mathematics, Drexel University, Philadelphia, PA 19104, USA}}
\email{dma68@drexel.edu}

\author[]{Anna L. Mazzucato}
\address{\scriptsize{Department of Mathematics, Penn State University, University Park, PA 16802, USA}}
\email{alm24@psu.edu}

\author[]{Riccardo Montalto} 
\address{\scriptsize{Dipartimento di Matematica, Universit\`a degli Studi di Milano, Via Saldini 50, I-20133, Milano, Italy}}
\email{riccardo.montalto@unimi.it}


\makeatletter
\@namedef{subjclassname@2020}{\textup{2020} Mathematics Subject Classification}
\makeatother
 


   
\begin{abstract}   
In this paper we analyze the Kuramoto-Sivashinsky equation (KSE), a model of flame-front propagation, on a two-dimensional square torus of arbitrary size $2 L$ with $L > \pi$. In this case, the linearized equation at the origin admits a finite number of growing modes, which corresponds to the positive eigenvalues of the linear operator $- \Delta^2 - \Delta$.
The problem of analyzing the long-time behavior of solutions of the  2D KSE in two spatial dimensions remains largely open; the only global existence results are for sufficiently small tori, or for sufficiently anisotropic and thin domains, due to the lack of good {\em a priori} estimates.
 The main purpose of this paper is to analyze the instability around growing modes at the nonlinear level. More precisely, we consider the maximal growing mode $\lambda_0$ and we show that there is a finite dimensional manifold of initial data of size $\e$ arbitrarily small such that the corresponding solutions become of size $O(1)$ over a time-scale of order $\log(\e^{- 1})$. The proof is based on several ingredients such as a sharp quantitative construction of an approximate solution bifurcating from the maximal linearly growing mode, a fixed point argument with exponential weights to construct local in time solutions, and a continuation argument based on sharp energy estimates and para-differential calculus. 
\end{abstract}

\maketitle

\tableofcontents

\section{Introduction}

In this work, we study the long-time dynamics of the Kuramoto-Sivashinsky equation  (KSE for short) in two space dimensions under periodic boundary conditions. The KSE is a model for long-wave instability in dissipative systems and was derived from a coordinate-free model of interface dynamics for flame-front propagation in combustion \cite{Kura,KT76,Siv77}. Because the KSE models the dynamics of an interface, the most physical dimensions are 1 and 2.

The KSE is a non-linear, hyperdissipative evolution equation that has both a scalar form for a potential $v$ with a non-linearity in conservative form, and a vector form for its gradient with a Burgers-type nonlinearity. 
We employ the scalar form of the equation that can be written as:
  $$
  \partial_t v + \Delta^2 v + \Delta v + \Pi(|\nabla  v|^2) = 0, \quad x =(x_1, x_2) \in \T^2_{L} :=  \T_{L} \times \T_{L},
  $$
  where 
  $$
  \T_{L} := \R /(2 L \Z) \simeq [- L, L], 
  $$
  and $\Pi v := v - \frac{1}{2 L} \int_{\T^2_L} v(x)\, d x $, $v : \R^2 \to \R$ is $(2L)$-periodic with respect to $x = (x_1, x_2).$  We note that we have chosen to use a square torus, 
  rather than allowing two independent dimensions such as $[0,L_{1}]\times[0,L_{2}].$  
  This is primarily for simplicity, so as not to over-complicate the proofs to be presented.
  An additional benefit, though, is that this choice also highlights how our results do not depend on any anisotropy in the problem. In particular,  we are not relying on any  thinness of the domain, 
  and thus our results  are genuinely two-dimensional.

  We shall assume that 
  \begin{equation}\label{condizione periodi}
 L  >  \pi\,,
  \end{equation}
  which gives rise to growing modes for the linearized operator.
  By rescaling the space variables $x$, we can set the problem on the standard torus $\T^2$. Namely, we rescale
  $$
  x \mapsto \omega x \quad \text{where} \quad   \omega:= \frac{\pi}{L},
  $$
  and we set 
$$
v(x) = \vphi (\omega  x),
$$
with 
$$
\vphi : \R^2 \to \R \quad 2\pi-\text{periodic in both variables.}
$$
The equation for $\vphi$ becomes 
\begin{equation}\label{KURAMOTO-RESCALED-VARIABLES}
\partial_t \vphi + \Delta^2_\omega \vphi + \Delta_\omega \vphi + \omega^2 \Pi (|\nabla \vphi|^2) = 0, \quad x = (x_1, x_2) \in \T^2,
\end{equation}
where 
$$
\Delta_\omega := \omega^2 \Delta, \quad \Delta_\omega^2 := \omega^4 \Delta^2 = \omega^4 ( \partial_{x_1}^2 + \partial_{x_2}^2 )^2\,.  
$$
Written in terms of frequency, $\omega,$ the condition \eqref{condizione periodi} 
becomes
\begin{equation}\label{condizione frequenze}
0 < \omega  < 1 \,. 
\end{equation}
We study the instability for the Cauchy problem 
\begin{equation}\label{KURAMOTO-RESCALED-VARIABLES2}
\begin{cases}
& \partial_t \vphi  = {\mathcal L}_\omega \vphi - \omega^2 \Pi (|\nabla \vphi|^2) , \quad x= (x_1, x_2) \in \T^2 \\
& \vphi(0, x) = \vphi_{in }(x),
\end{cases}
\end{equation}
where 
$$
{\mathcal L}_\omega := -  \Delta_\omega^2 - \Delta_\omega\,. 
$$
More precisely, we shall prove that for any $0 < \e \ll 1$ arbitrarily small,  there exists a finite dimensional manifold of initial data of size $O(\e)$ in the Sobolev space $H^s$, for which there exists a unique solution $\vphi$ up to time $T_\e \simeq \log(\e^{- 1})$ such that $\| \vphi(T_\e) \|_{H^s} = O(1)$. For this result to hold,  $s$ has to be taken large enough, specifically bigger than $4$. We shall prove this result in the general case in which the linearized equation can have an arbitrary number of growing modes (positive eigenvalues of the linear operator ${\mathcal L}_\omega$). The only hypothesis on the size of the torus, i.e. \eqref{condizione periodi}, guarantees that there is at least one strictly positive eigenvalue of ${\mathcal L}_\omega$, as already mentioned. 

The problem of studying the long-time dynamics of the Kuramoto-Sivashinsky equation (KSE) is far from being understood in dimension greater than 1. A main difficulty in proving global existence is the lack of a maximum principle for the KSE, due to the presence of the biharmonic operator. In dimension $d=1$, it is possible to obtain an {\em a priori} estimate on the $L^2$ norm of $u$ due to the special structure of the non-linearity, which can be written in divergence form. The $L^2$ control  allows to prove global existence by a continuation argument \cite{Tad86}.  The one-dimensional global existence result of \cite{Tad86} allows exponential growth of solutions, because the linearized operator
$\partial_t +\Delta^2+\Delta$ admits exponentially growing modes for large enough periods. 
It has then been demonstrated that this exponential growth does not actually occur, and solutions which start near the zero solution remain near the zero solution\cite{Good94}, \cite{NST85}; this is sometimes described as stability of the
zero solution in the one-dimensional KSE, but this is not stability in the usual sense in that it is not possible to control solutions to remain arbitrarily close to the zero
solution. 

For many well-studied equations such as the incompressible
Navier-Stokes or Euler equations, the availability of an
$L^2$ estimate is independent of dimension.  This situation
is markedly different for the KSE.
In fact, there is no known global estimate for any $L^p$ norm in $d>1$, and this is a primary reason why the issue of global existence of solutions to the KSE is still open. 

There is an extensive literature concerning the KSE in dimension $d=1$.  While we do not survey the full literature in the one-dimensional case, we mention results concerning analyticity of solutions (see  \cite{CEES93,Gru00} and references therein), optimal bounds on the growth of the $L^2$ norm as a function of the period $L$ \cite{BG06,GO05, GJO15, GF19, Otto09, Stan, Stan2, Stan3} (see also \cite{SSArXiv07}), and inertial manifolds
\cite{FNST88, JKT90}. 

There are few results for the higher-dimensional KSE.  Short-time existence and analyticity is known to hold in the full space with data in certain $L^p$ spaces \cite{BS07} (see also \cite{SSArXiv07}). For $d=2$, global existence holds for thin or sufficiently anisotropic domains \cite{BKRZ14, KM23, Mol00b, SellT92}. A first isotropic result was \cite{AM19}, in which Ambrose-Mazzucato studied the KSE on  $\T^2$ and proved short-time existence and analyticity with a bound on the analyticity radius for data in $L^2$ and in the Wiener algebra. They established global existence
of mild solutions for mean-free data sufficiently small in $L^2$ and in the Wiener algebra, but only in the absence of growing modes, which happens in a two-dimensional torus $[0,L_{1}]\times[0,L_{2}]$ when $L_1,\;
 L_2<2 \pi$. Subsequently, Ambrose and Mazzucato extended this to global existence of small solutions 
 when $L_{1}$ and $L_{2}$ are each slightly larger than $2\pi,$ so that there is a linearly growing mode
 present in both spatial directions \cite{AM21}; this is the only isotropic global result for the higher-dimensional 
 KSE which allows a linearly growing mode in each spatial direction.
 In \cite{FM22}, Feng-Mazzucato proved existence of mild solutions for sufficiently small data $\phi(0)$ in $L^2$ on an arbitrary interval of time $[0,T]$  in the absence of growing modes. In \cite{ALN24}, solutions of the KSE in arbitrary dimension are proved to exist when
 the initial data has low regularity, but the solutions are only global in the absence of linearly 
 growing modes.  We also mention a detailed numerical study of the two-dimensional KSE \cite{KKP15}.
 
We only consider the classical form of KSE, and not any modified or generalized KSE, for which more results are known (see e.g. \cite{GMP08, Mol00}).
One recent modified equation is the anisotropically reduced KSE \cite{LY20}, in which a maximum principle holds for
one component of the gradient of the solution; another anisotropic modification
of the KSE is studied in \cite{TKP18}.
In \cite{FM22}, Feng and Mazzucato also considered the KSE with additional linear advection and prove that global existence can be achieved in the presence of growing modes and for arbitrary data, if the advecting field  is relaxation enhancing with sufficiently small dissipation time. We also mention \cite{CZDFM21} where the advection term is given by a shear flow. Another work on equations similar in some ways to the KSE 
is \cite{BBT03}, in which the question
of global existence or singularity formation was studied for hyperviscous equations of Hamilton-Jacobi-type. We also mention the work \cite{KOS-TITI} where a 1D KSE-type equation with a dispersive term is considered, showing that the dispersive term weakens the global bounds.

In the present paper, we study the KSE in the case of the two dimensional periodic box $\T^2_L := (\R/(2L \Z))^2$ with $L > \pi$, so that we allow the presence of an arbitrary large number of linearly growing modes. We rigorously prove instability, namely, we construct initial data of arbitrarily small  size $O(\e)$ such  that in finite time $T \simeq \log(\e^{- 1})$ the corresponding solutions become of order $O(1)$ (in the $H^s$ topology). Our approach is inspired by the technique developed in \cite{Grenier},  in which the author proves some nonlinear instability results for the Euler and Prandtl equations (see also \cite{GN19}).

Throughout the paper, we use standard notation to represent function spaces. In particular, $H^s(\T^d)$, $s\in \R$, is the standard $L^2$-based Sobolev space equipped with the norm 
$$
\| u \|_{H^s} := \Big( \sum_{\xi \in \Z^d} \langle \xi \rangle^{2 s} |\widehat u(\xi)|^2 \Big)^{\frac12}, \quad \langle \xi \rangle := (1 + |\xi|^2)^{\frac12},
$$
and we denote by $H^s_0(\T^d)$ the space of $H^s$ functions with zero average. For any $T > 0$, we denote by ${\mathcal C}^0([0, T], H^s_0(\T^d))$ the space of continuous functions 
$$
[0, T] \to H^s_0(\T^d)\,, \quad t \mapsto u(t),
$$ 
equipped by the sup norm
$$
\| u \|_{{\mathcal C}^0_T H^s_x} := \sup_{t \in [0, T]} \| u(t) \|_{H^s}\,. 
$$
For any $k \geq 1$ integer we also denote by ${\mathcal C}^k([0, T], H^s_0(\T^d))$ the space of the ${\mathcal C}^k$ maps
$$
[0, T] \to H^s_0(\T^d)\,, \quad t \mapsto u(t),
$$
equipped by the norm 
$$
\| u \|_{{\mathcal C}^k_T H^s_x} := \sum_{n = 0}^k \| \partial_t^n u \|_{{\mathcal C}^0_T H^s_x}\,. 
$$ 
Given two Banach spaces $X$ and $Y$, we denote by ${\mathcal B}(X, Y)$ the space of bounded linear operators from $X$ to $Y$ equipped with the standard operator norm $\| \cdot \|_{{\mathcal B}(X, Y)}$. If $X = Y$, we use the notation ${\mathcal B}(X) \equiv {\mathcal B}(X, X)$. Given ${\mathcal H}$ a Hilbert space and a linear operator $A : {\mathcal D}(A)\to {\mathcal H}$ where ${\mathcal D}(A)$ is a dense subspace of ${\mathcal H}$, we denote by $A^*$ the standard adjoint operator of $A$. 

\noindent
Given some parameters $a_1, \ldots, a_n > 0$, we use the following notation. We write
$$
A \lesssim_{a_1, \ldots, a_n} B 
$$
if there exists a positive constant $C(a_1, \ldots, a_n) > 0$ such that 
$$
A \leq C(a_1, \ldots, a_n) B\,. 
$$
\section{Statement of the main result and ideas of the proof}

We expand a function $u \in L^2_0(\T^2)$ in Fourier series as 
$$
u (x) \equiv u(x_1, x_2) = \sum_{(\xi_1, \xi_2) \in \Z^2 \setminus \{ 0 \}} \widehat u(\xi_1, \xi_2) e^{i(x_1 \xi_1 + x_2 \xi_2)} = \sum_{\xi \in \Z^2 \setminus \{ 0 \}} \widehat u(\xi) e^{\ii x \cdot \xi}. 
$$
(We recall that $L^2_0(\T^2)$  is the subspace of mean-free elements, therefore $ \widehat u(0) = 0$.) 

We  consider the linear operator ${\mathcal L}_\omega =-  \Delta_\omega^2 - \Delta_\omega$. For any $\xi = (\xi_1, \xi_2) \in \Z^2 \setminus \{ 0 \}$, one has that 
\begin{equation}\label{autovalori 1}
\begin{aligned}
{\mathcal L}_\omega (e^{\ii x \cdot \xi} ) & = \lambda_\omega(\xi) e^{\ii x \cdot \xi}  \,, \\
\lambda_\omega (\xi) & := - \omega^4  |\xi|^4  + \omega^2 |\xi|^2      = - \omega^2 |\xi|^2 ( \omega^2 |\xi|^2 - 1 ) \,, \quad \xi \in \Z^2 \setminus \{ 0 \}\,. 
\end{aligned}
\end{equation}
The spectrum of ${\mathcal L}_\omega$ is then given by 
$$
\sigma({\mathcal L}_\omega) = \Big\{ \lambda_\omega(\xi) = - \omega^2 |\xi|^2 ( \omega^2 |\xi|^2 - 1 ) : \xi \in \Z^2 \setminus \{ 0 \} \Big\}\,. 
$$
Since  $0 < \omega <1$, one has that 
\begin{equation}\label{one positive eigenvalue}
\lambda_\omega(1, 0) = \lambda_\omega(0, 1) =  \omega^2(1 - \omega )(\omega + 1) > 0.
\end{equation}
Therefore there is at least one strictly positive eigenvalue of ${\mathcal L}_\omega$. We split the spectrum $\sigma({\mathcal L}_\omega)$ of the linear operator ${\mathcal L}_\omega$ as
$$
\begin{aligned}
& \sigma({\mathcal L}_\omega) = \sigma_+({\mathcal L}_\omega) \cup \sigma_-({\mathcal L}_\omega)\,, \\
& \sigma_+({\mathcal L}_\omega) := \Big\{ \lambda \in \sigma({\mathcal L}_\omega) : \lambda \geq 0 \Big\}\,, \\
& \sigma_-({\mathcal L}_\omega) :=  \Big\{ \lambda \in \sigma({\mathcal L}_\omega) : \lambda < 0 \Big\}\,. 
\end{aligned}
$$
By \eqref{one positive eigenvalue}, one has that $\sigma_+({\mathcal L}_\omega) \neq \emptyset$ and there is at least one eigenvalue in $\sigma_+({\mathcal L}_\omega)$ that is strictly positive. A direct calculation shows that 
\begin{equation}\label{sigma + - spettro 2}
\begin{aligned}
\sigma_+({\mathcal L}_\omega) & = \Big\{ \lambda_\omega(\xi) : 0 < |\xi| \leq \frac{1}{\omega}\Big\}\,, \\
\sigma_-({\mathcal L}_\omega) & = \Big\{ \lambda_\omega(\xi) :  |\xi| > \frac{1}{\omega}\Big\}.
\end{aligned}
\end{equation}
Hence $\sigma_+({\mathcal L}_\omega)$ is non-empty, finite and $\sigma_-({\mathcal L}_\omega)$ is infinite. Moreover there could be some integer vectors $\xi \in \Z^2 \setminus \{ 0 \}$ such that $|\xi| = \frac{1}{\omega}$. For such vectors $\xi$ one has that $\lambda_\omega(\xi) = 0$. We denote by $\lambda_0$ the maximum (strictly positive) eigenvalue of ${\mathcal L}_\omega$, namely,
\begin{equation}\label{def lambda 0}
\lambda_0 := {\rm max} \,\sigma_+({\mathcal L}_\omega) = {\rm max}\, \sigma({\mathcal L}_\omega) > 0\,. 
\end{equation}
Then 
$$
\lambda_0 = \lambda_\omega(\xi_0) \quad \text{for some} \quad \xi_0 \in \Z^2 \setminus \{ 0 \} \quad \text{with} \quad 0 < |\xi_0| < \frac{1}{\omega}\,. 
$$
The eigenspace associated to $\lambda_0$ is given by 
$$
E(\lambda_0) := {\rm span}\big\{ e^{ \ii x \cdot \xi} : |\xi| = |\xi_0| \big\}\,. 
$$

We focus on initial conditions for  \eqref{KURAMOTO-RESCALED-VARIABLES2} that are Fourier localized around 
wavenumber $|\xi_0|$; that is, we take  $\vphi_{in}(x) = \sum_{|\xi| = |\xi_0|} \widehat \vphi_{in}(\xi) e^{\ii x \cdot \xi} $, with $\widehat \vphi_{in}(\xi) = \overline{\widehat \vphi_{in}(- \xi)}$, which guarantees that $\vphi_{in}$ is real valued.  We then have that ${\mathcal L}_\omega \vphi_{in} = \lambda_0 \vphi_{in}$.  For future use, for any $n \geq 2$ we also introduce  the finite dimensional spaces
\begin{equation}\label{spaces En n geq 2}
E_n := {\rm span}\Big\{ e^{\ii x \cdot \xi} : 1 \leq |\xi| \leq n |\xi_0| \Big\}\,.
\end{equation}
We will study the Cauchy problem:
\begin{equation}\label{Prob Cauchy instab}
\begin{cases}
\partial_t \vphi = {\mathcal L}_\omega \vphi - \omega^2 \Pi (|\nabla \vphi|^2), \\
\vphi(0, x) = \e \vphi_{in}(x),
\end{cases}
\end{equation}
where $0 < \e \ll 1$ is small enough. The main result of this paper is the following theorem.

\begin{thm}\label{main theorem statement}
Let $s \geq 4$ and let $\omega$ satisfy \eqref{condizione frequenze}. There are two constants $\delta_0 \equiv \delta_0(s, \omega), \e_0 \equiv \e_0(s, \omega) > 0$ such that for any $\e \in (0, \e_0)$, for any initial  $\e \vphi_{in}$ with $\vphi_{in} \in  E(\lambda_0)$ normalized in $H^s$, i.e. $\| \vphi_{in } \|_{H^s} = 1$, there exist a time $T_\e$ with 
$$ 
\frac{1}{2 \lambda_0} \log(\e^{- 1})  < T_\e < \frac{1}{ \lambda_0} \log(\e^{- 1}),
$$ 
and a unique solution $\vphi \in {\mathcal C}^0([0, T_\e], H^s_0(\T^2)) \cap {\mathcal C}^1([0, T_\e], H^{s - 4}_0(\T^2))$ of the Cauchy problem \eqref{Prob Cauchy instab} satisfying 
$$
\| \vphi(T_\e) \|_{H^s} \geq \delta_0\,. 
$$
\end{thm}

We must compare this result to the global existence result of \cite{AM21}. 
With the present notation, for
$\omega$ quite close to $1$ and satisfying \eqref{condizione frequenze}, and with small initial 
data, the result of of \cite{AM21} applies.  Then, solutions exist for all time
and remain small, proportional to the size of the initial data, which is consistent with Theorem \ref{main theorem statement}.  
Indeed, as $\omega\rightarrow1^{-},$ we would have $\e_{0}\rightarrow0$ and
correspondingly $\delta_{0}\rightarrow0$ as well.

The next subsection is devoted to illustrating the ideas of the proof of Theorem \ref{main theorem statement}. 

\medskip

\subsection{Ideas of the proof}\label{sotto-sezione-idee-dim}
\noindent
We consider initial data $\e \vphi_{in}$, arbitrarily small, i.e. $0 < \e \ll 1$ where $\vphi_{in}$ is an eigenfunction of ${\mathcal L}_\omega$, normalized in $H^s$, associated to the biggest (positive) eigenvalue of ${\mathcal L}_\omega$, i.e.
$$
{\mathcal L}_\omega \vphi_{in} = \lambda_0 \vphi_{in}, \quad \| \vphi_{in} \|_{H^s} = 1\,. 
$$
We note that neglecting the nonlinearity, the linear problem
$$
\begin{cases}
\partial_t \vphi_1 = {\mathcal L}_\omega \vphi_1 \\
\vphi_1(0, x) = \e \vphi_{in}(x)
\end{cases}
$$
admits the solution 
$$
\vphi_1(t, x) = \e e^{\lambda_0 t} \vphi_{in}(x)\,. 
$$
Hence $\| \vphi_1(0) \|_{H^s} = \e \ll 1$. First, we explore when this linear solution becomes of order $O(1)$ in $H^s$, namely, for which time $T_0 \equiv T_0(\e)$ one has that 
$$
\| \vphi_1(T_0) \|_{H^s} = \e e^{\lambda_0 T_0} =  1\,. 
$$
Clearly this happens when 
$$
T_0 := \frac{1}{\lambda_0} \log(\e^{- 1})\,. 
$$
Informally, the goal is then to extend this very simple mechanism for the linearized equation to the nonlinear problem. We shall prove that there exists a time $T \simeq \frac{1}{\lambda_0} \log(\e^{- 1})$, more precisely $ T_0/2 < T < T_0$, and a unique solution of \eqref{Prob Cauchy instab} such that $\| \vphi(0) \|_{H^s} = \e$ and $\| \vphi(T)\|_{H^s} = O(1)$. This ``instability time'' $T$ is strictly smaller than the ``instability time'' $T_0$ given by the linearized problem, an effect of the nonlinear term appearing in the KSE.  We look for such an instability time $T$ as
$$
T \equiv T_\sigma \equiv T_\sigma(\e) = T_0 - \sigma,
$$
where $\sigma \equiv \sigma(s, \omega) \gg 0$ has to be chosen large enough, but independently of $\e$. We observe that given $\sigma > 0$, 
$$
T_0/2 < T_\sigma < T_0
$$
is satisfied provided that 
$$
 0 < \e <  \e_0 := e^{- 2 \lambda_0 \sigma}\,.
$$
We stress that this is the only condition that we require on $\e$ along the whole paper.

\noindent
We construct the solution $\vphi(t, x)$ in such a way that, over the time interval $[0, T_\sigma]$, it holds $\vphi(t, x) \simeq \e e^{\lambda_0 t} \vphi_{in}(x)$. More precisely, we prove the quantitative estimate 
$$
\sup_{t \in [0, T_\sigma]}\| \vphi(t) - \e e^{\lambda_0 t} \vphi_{in}\|_{H^s} \leq C(s, \omega) e^{- 2 \lambda_0 \sigma}\,.
$$
This estimate implies that, for $\sigma \gg 0$ large enough, but independent of $\e$, one has that 
$$
\| \vphi(T_\sigma) \|_{H^s} \geq \delta_0 := \frac12 e^{- \lambda_0 \sigma}\,. 
$$
In order to prove this estimates, the proof is essentially divided into three steps:
\begin{enumerate}
\item{\bf Construction of an approximate solution $\vphi_{app}$ close to the linear solution $\e e^{\lambda_0 t} \vphi_{in}$;}
\item{\bf Local in time existence of a solution close to the approximate solution $\vphi_{app}$;}
\item{\bf Continuation argument up to time $T_\sigma \simeq \frac{1}{\lambda_0} \log(\e^{- 1})$ and instability estimate.}
\end{enumerate}
In the following, we shall describe these three steps. 

\bigskip

\noindent
{\bf Step 1. Construction of the approximate solution $\vphi_{app}$.} In Section \ref{section-approx-solution}, we will construct a ${\mathcal C}^\infty$ approximate solution at any order of the Cauchy problem \eqref{Prob Cauchy instab} as a regular expansion in power of  $\e$. That is,  we look for 
$$
\vphi_{app}(t, x) = \sum_{n = 1}^N \e^n \vphi_n(t, x),
$$
where $N\in \mathbb{N}$ is arbitrary.
By inserting this ansatz into the equation, one finds that $\vphi_1, \ldots, \vphi_N$ can be recursively determined by solving linear equations. Actually $\e \vphi_1$ is the solution of the linearized problem with initial data $\e \vphi_{in}$, whereas for $n \geq 2$, $\vphi_n(t, x)$ solves the linear inhomogeneous Cauchy problem \eqref{eq induttiva per vphi n}. The key point is to quantify the exponential growth in time of all the $\vphi_n$, $n = 2, \ldots, N$ with sharp bounds.  We shall prove that 
$$
\| \vphi_n(t) \|_{H^s} \lesssim_{s, n, \omega} e^{n \lambda_0 t}, \quad \forall t \geq 0, \quad \forall s \geq 0, \quad n = 1, \ldots, N.
$$
We exploit the following property of the  quadratic nonlinearity: if $u$ is a function supported on the Fourier modes $|\xi| \leq N$ and $v$ is supported on the Fourier modes $|\xi| \leq M$, then the nonlinearity ${\mathcal Q}(u, v)$ is supported on the Fourier modes $|\xi| \leq N + M$. The estimated above is proved in Lemma \ref{stime vphi n}, by performing an induction argument  of the Fourier coefficients of $\vphi_n$, by using variation of the constants and using that $\lambda_0$ is the maximum positive eigenvalue of the linear operator ${\mathcal L}_\omega$.  The  error term $r_N(t, x)$ is then bounded in Lemma \ref{lemma stima rN}, where it is proved that 
$$
\| r_N(t) \|_{H^s} \lesssim_{s, N, \omega} \e^{N + 1} e^{(N + 1) \lambda_0 t}, \quad \forall t \in [0, T_\sigma], \quad \forall s \geq 0\,.
$$
In deriving these estimates,  the elementary fact that, since $\e e^{\lambda_0 T_0} = 1$,  for any $0 \leq t \leq T_\sigma = T_0 - \sigma$
$$
\e e^{\lambda_0 t} \leq e^{- \lambda_0 \sigma} \ll 1 \quad \text{for} \quad \sigma \gg 0 \quad \text{large enough,}
$$
is used repeatedly.
\bigskip

\noindent
{\bf Step 2. Local in time existence close to $\vphi_{app}$.} In Section \ref{section-local-existence}, we will prove  existence up to $O(1)$ times  of solutions $\vphi(t, x) \simeq \vphi_{app}(t, x)$. To this end, we write 
$$
\vphi = \vphi_{app} + u
$$
and we solve the Cauchy problem \eqref{PDE app sol u}, that is, 
$$
\begin{cases}
\partial_t u = {\mathcal L}_\omega u + 2 {\mathcal Q}(\vphi_{app}, u) + {\mathcal Q}(u, u) + r_N \\
u(0, x) = 0\,. 
\end{cases}
$$
In Proposition \ref{teo esistenza locale}, we prove existence of smooth solutions up to time $T=O(1)$ by a fixed point argument in a weighted space that keeps track of the time exponential growth due to the maximal unstable eigenvalue $\lambda_0$. More precisely, we employ the following norm (see \eqref{palla punto fisso}):
$$
\| u \|_{s, T} := \sup_{t \in [0, T]} e^{- N \lambda_0 t} \| u(t) \|_{H^s}.
$$
The local solution $u$ then will satisfy
\begin{equation}\label{stima locale u intro}
\| u(t) \|_{H^s} \leq \e^N e^{N \lambda_0 t}, \quad \forall t \in [0, T]\,. 
\end{equation}
The solution is a ``mild'' solution, that is, a fixed point of the map $\Phi$ given in \eqref{def mappa Phi punto fisso}. 
As customary for semilinear parabolic equation, control of the integral term that contains the nonlinearity is obtained through smoothing estimates for  the semigroup $e^{t {\mathcal L}_\omega}$, $t>0$, specifically that that gives a gain of one derivative (see Lemma \ref{lemma stime semigruppo libero}). This bound degenerates at the origin as for the biharmonic operator, that is, $\sim \frac{1}{t^{\frac14}}$.   In order to estimate all the terms appearing in the Duhamel formula with respect to the norm $\| \cdot \|_{s, T}$,  we establish quantitative bounds involving the error term $r_N$ and the nonlinearity  in Lemmata \ref{stima propagatore con rN}, \ref{lemma stime parte quadratica}. Finally, the fixed point argument on $\Phi$ is implemented in Lemma \ref{lemma finale punto fisso Phi}. 

\bigskip

\noindent
{\bf Step 3. Continuation argument up to  time $T_\sigma \simeq \frac{1}{\lambda_0} \log(\e^{- 1})$ and instability estimate.}

\noindent
In Section \ref{sezione long time existence}, we discuss how to extend the local solution $u$ of the Cauchy problem \eqref{PDE app sol u}, satisfying \eqref{stima locale u intro} , up to  time $T_\sigma \simeq \frac{1}{\lambda_0} \log(\e^{- 1})$. This  argument is  by contradiction. We call $T_*$ the maximal time for which $u$ is the solution of \eqref{PDE app sol u} satisfying \eqref{stima locale u intro}. We then assume by contradiction that $T_* \leq T_\sigma$ and we perform a sharp energy estimate over the time interval $[0, T_*]$. Since the nonlinearity contains derivatives, we use the theory of para-differential operators in order to avoid a possible loss of derivatives in the energy estimates. We collect all the facts needed on para-differential operators in Appendix \ref{appendice paradiff}. First of all, we paralinearize \eqref{PDE app sol u} and we reduce it to the para-differential problem \eqref{prob cauchy per stima di energia 1}, which is
\begin{equation}\label{paralin problem intro}
\begin{cases}
\partial_t u = {\mathcal L}_\omega u + T_{A} \cdot \nabla u + f_N(t)  \\
u(0, x) = 0
\end{cases}
\end{equation}
where $A(t)$ and $f_N(t)$ (that clearly depend on $u$) satisfy the estimates 
$$
\begin{aligned}
& \| A(t) \|_{H^{s - 1}} \lesssim_s \e e^{\lambda_0 t}  \,, \quad \forall t \in [0, T_*]\\
& \| f_N(t) \|_{H^s} \lesssim_{s, N, \omega} \e^{N + 1} e^{(N + 1) \lambda_0 t}, \qquad \forall t \in [0, T_*]
\end{aligned}
$$
which is true under our hypothesis that $s \geq 4$. The para-product operator $T_a$, given a suffciently regular Sobolev function $a$, is defined in \eqref{def paraproduct}. Then by performing an energy estimate on \eqref{paralin problem intro}, we obtain that 
$$
\frac{d}{d t} \| u(t) \|_{H^s}^2 \leq \lambda_0 \| u(t) \|_{H^s}^2 + C_0 \e^{2 N + 1} e^{(2 N + 1)\lambda_0 t} \,, \quad \forall t \in [0, T_*]
$$
for some constant $C_0 \equiv C_0(s, N, \omega) > 0$ large enough. It follows that
$$
\| u(t) \|_{H^s} \leq C \e^{N + \frac12} e^{(N + \frac12) \lambda_0 t}, \quad \forall t \in [0, T_*]
$$
for some constant $C \equiv C(s, N, \omega) > 0$. Next, using that $C(\e e^{\lambda_0 t})^{\frac12} \leq Ce^{- \frac{\lambda_0 \sigma}{2}}  \leq \frac12$ (for $\sigma \gg 0$ large enough) gives
$$
\| u(t) \|_{H^s} \leq \frac12 \e^N e^{N \lambda_0 t}, \quad \forall t \in [0, T_*]\,.
$$
We now reach  a contradiction, since by the definition of $T_*$ we should have instead:
$$
\sup_{t \in [0, T_*]} e^{- N \lambda_0 t}\| u(t) \|_{H^s} = \e^N\,. 
$$
Therefore $T_* > T_\sigma$ and $\| u(t) \|_{H^s} \leq e^{N \lambda_0 t} \e^N$ for any $t \in [0, T_\sigma]$. Finally, in the last section, Section \ref{sezione finale dim teorema}, we prove that the true solution $\vphi$ is a small perturbation of $\e e^{\lambda_0 t} \vphi_{in}$ over the time interval $[0, T_\sigma]$ and we show that $\| \vphi(T_\sigma) \|_{H^s} \geq \frac12e^{- \lambda_0 \sigma}$, for $\sigma \gg 0$ large enough and independent of $\e$. We can make the choice $N = 2$. This concludes the proof of our main result, Theorem \ref{main theorem statement}. 

\bigskip

\section{Construction of an approximate solution}\label{section-approx-solution}
We look for an approximate solution of the Cauchy problem \eqref{Prob Cauchy instab} of the form
\begin{equation}\label{approx-solution}
\vphi_{app}(t, x) = \sum_{n = 1}^N \e^n \vphi_n(t, x) \quad \text{such that} \quad \vphi_{app}(0, x) = \e \vphi_{in}(x)\,. 
\end{equation}
We set
$$
\vphi_1(t, x) := e^{\lambda_0 t } \vphi_{in}(x)\,.
$$
Clearly 
$$
\begin{cases}
& \partial_t \vphi_1 = {\mathcal L}_\omega \vphi_1, \\
& \vphi_1(0) = \vphi_{in}.
\end{cases}
$$
We will determine $\vphi_2, \vphi_3, \ldots, \vphi_N$ inductively. 
We begin by rewriting equation \eqref{KURAMOTO-RESCALED-VARIABLES} as 
\begin{equation}\label{KURAMOTO-RESCALED-VARIABLES-2}
\partial_t \vphi = {\mathcal L}_\omega \vphi + {\mathcal Q}(\vphi, \vphi),
\end{equation}
where the bilinear form $(\vphi_1, \vphi_2) \mapsto {\mathcal Q}(\vphi_1, \vphi_2)$ is defined by 
\begin{equation}\label{nonlinearity mathcal Q}
{\mathcal Q}(\vphi_1, \vphi_2) :=- \omega^2 \Pi \big( \nabla  \vphi_1 \cdot \nabla \vphi_2 \big)\,. 
\end{equation}
Inserting  the ansatz \eqref{approx-solution} into \eqref{KURAMOTO-RESCALED-VARIABLES-2} gives that
$$
\sum_{n = 1}^N \e^n \partial_t \vphi_n = \sum_{n = 1}^N \e^n {\mathcal L}_\omega \vphi_n + \sum_{n = 2}^N \e^n \sum_{k = 1}^{n - 1} {\mathcal Q}(\vphi_{n - k}\,,\, \vphi_k) + r_N(t, x).
$$
Using the equation satisfied by $\vphi_1$, we can rewrite the above expression as
$$
\sum_{n = 2}^N \e^n \partial_t \vphi_n = \sum_{n = 2}^N \e^n {\mathcal L}_\omega \vphi_n + \sum_{n = 2}^N \e^n \sum_{k = 1}^{n - 1} {\mathcal Q}(\vphi_{n - k}\,,\, \vphi_k) + r_N(t, x)\,,
$$
where the ``error'' $r_N$ is given by 
\begin{equation}\label{errore soluzioni appprossimate}
r_N := \sum_{n = N + 1}^{2 N} \e^n \sum_{k = 1}^{n - 1} {\mathcal Q}(\vphi_{n - k}\,,\, \vphi_k)\,. 
\end{equation}
We then determine recursively $\vphi_2, \vphi_3, \ldots, \vphi_N$ by solving 
\begin{equation}\label{eq per vphi 2}
\partial_t \vphi_2 = {\mathcal L}_\omega \vphi_2 + {\mathcal Q}(\vphi_1, \vphi_1), \quad \vphi_2(0, x) = 0
\end{equation}
and, assuming  $\vphi_2, \ldots, \vphi_{n - 1}$ given,  by solving 
\begin{equation}\label{eq induttiva per vphi n}
\partial_t \vphi_n = {\mathcal L}_\omega \vphi_n + \sum_{k = 1}^{n - 1} {\mathcal Q}(\vphi_{n - k}\,,\, \vphi_k) \,, \quad \vphi_n(0, x) = 0\,. 
\end{equation}
In order to do precise induction estimates on $\vphi_2, \ldots, \vphi_N$ we prove the following two preliminary lemmata.

\begin{lemma}\label{lemma variazione costanti En}
Let $n \geq 2$. Let  $f: [0, + \infty) \to E_n$, $t\mapsto f(t)$, be a smooth map, where  $f(t)(x) := f(t, x) = \sum_{1 \leq |\xi| \leq n|\xi_0|} \widehat f(t, \xi) e^{\ii x \cdot \xi}$. Assume that there is a constant $C(f) > 0$ such that  
$$
|\widehat f(t, \xi)| \leq C(f) e^{n \lambda_0 t}, \quad \forall t \geq 0, \quad \forall 1 \leq |\xi| \leq n|\xi_0|\,. 
$$
 Then the solution $u(t, x)$ of
\begin{equation}\label{forced system cal L omega}
\begin{cases}
\partial_t u = {\mathcal L}_\omega u + f(t, x)\,, \\
u(0, x) = 0
\end{cases}
\end{equation}
is of the form $u(t, x) = \sum_{1 \leq |\xi| \leq n |\xi_0|} \widehat u(t, \xi) e^{\ii x \cdot \xi}$ and 
$$
|\widehat u(t, \xi)| \leq \frac{2 C(f)}{(n - 1) \lambda_0} e^{\lambda_0 n t}, \quad \forall t \geq 0 \,, \quad \forall 1 \leq |\xi| \leq n |\xi_0|\,. 
$$
\end{lemma}

\begin{proof}
Since the operator ${\mathcal L}_\omega$ is a Fourier multiplier and  $f(t, \cdot) \in E_n$, we look for $u(t, \cdot ) \in E_n$ solving \eqref{forced system cal L omega}\,. Therefore, for any $1 \leq |\xi| \leq n$  we  let $\widehat u(t, \xi)$ solve
$$
\begin{cases}
\partial_t \widehat u(t, \xi) = \lambda_\omega(\xi) \widehat u(t, \xi) + \widehat f(t, \xi), \\
\widehat u(0, \xi) = 0\,.
\end{cases}
$$
By the variation of constants, we have that
$$
\widehat u(t, \xi) = \int_0^t e^{\lambda_\omega(\xi)(t - \tau)} \widehat f(\tau, \xi)\, d \tau.
$$
It follows that
$$
\begin{aligned}
|\widehat u(t, \xi)| & \leq \int_0^t  e^{\lambda_\omega(\xi)(t - \tau)} |\widehat f(\tau, \xi)|\, d \tau  \leq C(f) \int_0^t  e^{\lambda_\omega(\xi)(t - \tau)} e^{\lambda_0 n \tau}\, d \tau \\
&  \leq C(f) e^{\lambda_\omega(\xi) t } \int_0^t e^{(n \lambda_0 - \lambda_\omega(\xi)) \tau}\, d \tau \\
& \leq \frac{C(f)}{n \lambda_0 - \lambda_\omega(\xi)} e^{\lambda_\omega(\xi) t }  \Big( e^{(n \lambda_0 - \lambda_\omega(\xi)) t} - 1\Big) \\
& \leq \frac{C(f)}{n \lambda_0 - \lambda_\omega(\xi)}  \Big( e^{n \lambda_0 t} + e^{\lambda_\omega(\xi) t } \Big) \,.
\end{aligned}
$$
We observe that  $\lambda_\omega(\xi) \leq \lambda_0$ for any $\xi \in \Z^2 \setminus \{ 0 \}$ by \eqref{def lambda 0}. This fact implies
$$
n \lambda_0 - \lambda_\omega(\xi) \geq (n - 1) \lambda_0 \quad \text{and} \quad e^{\lambda_\omega(\xi) t } \leq e^{\lambda_0 t} \leq e^{n \lambda_0 t},
$$
given that $n\geq 2$. Consequently, the following bound holds:
$$
|\widehat u(t, \xi)| \leq \frac{2 C(f)}{(n - 1) \lambda_0} e^{n \lambda_0 t}\,. 
$$
\end{proof}

\begin{lemma}\label{lemma astratto stima forcing}
Let $m, n \geq 1$ be two integers. Assume that  $\vphi_n:[0, + \infty) \to E_n$, $t\mapsto \vphi(t)$,  $\vphi_m: [0, + \infty) \to E_m$, $t\mapsto \phi_m(t)$, and  assume that 
$$
\begin{aligned}
& |\widehat \vphi_n(t, \xi)| \leq C(\vphi_n) \frac{e^{n \lambda_0 t}}{\lambda_0^{n - 1}}, \quad \forall t \geq 0\,, \quad \forall 1 \leq |\xi| \leq n|\xi_0|\,, \\
& |\widehat \vphi_m(t, \xi)| \leq C(\vphi_m) \frac{e^{m \lambda_0 t}}{\lambda_0^{m - 1}}, \quad \forall t \geq 0\,, \quad \forall 1 \leq |\xi| \leq m|\xi_0|\
\end{aligned}
$$
for some constants $C(\vphi_n), C(\vphi_m) >0$. Then $f := {\mathcal Q}(\vphi_n, \vphi_m) \in E_{n + m}$ satisfies 
$$
|\widehat f(t, \xi)| \lesssim_{n, m} C(\vphi_n) C(\vphi_m) \frac{e^{\lambda_0 (n + m) t}}{\lambda_0^{n + m - 2}}\,, \quad \forall t \geq 0, \quad \forall 1 \leq |\xi| \leq (n + m) |\xi_0|\,. 
$$
\end{lemma}

\begin{proof}
By linearity in each entry, $f$ can be expanded as 
\begin{equation}\label{penn state 0}
f = {\mathcal Q}(\vphi_n, \vphi_m) = \sum_{\begin{subarray}{c}
1 \leq |\xi_1| \leq n |\xi_0| \\
1 \leq |\xi_2| \leq m |\xi_0|
\end{subarray}} \widehat \vphi_n(t, \xi_1) \widehat \vphi_m(t, \xi_2) {\mathcal Q}\Big(e^{\ii x \cdot \xi_1}\,,\, e^{\ii x \cdot \xi_2}\Big)\,. 
\end{equation}
Next, we observe that for any $\xi_1, \xi_2 \in \Z^2 \setminus \{ 0 \}$, 
$$
\begin{aligned}
{\mathcal Q}\Big( e^{\ii x \cdot \xi_1}\,,\, e^{\ii x \cdot \xi_2} \Big) & =-  \omega^2 \Pi \Big( \nabla (e^{\ii x \cdot \xi_1}) \cdot \nabla(e^{\ii x \cdot \xi_2}) \Big)  =  \omega^2 \xi_1 \cdot \xi_2 \Pi \Big( e^{\ii x \cdot (\xi_1 + \xi_2)}  \Big)\,. 
\end{aligned}
$$
It then follows from \eqref{penn state 0} that
$$
f(t, x) = \sum_{1 \leq |\xi| \leq (n + m) |\xi_0|} \widehat f(t, \xi) e^{\ii x \cdot \xi},
$$
where for any $1 \leq |\xi| \leq (n + m) |\xi_0|$, 
$$
\widehat f(t, \xi) := \sum_{\begin{subarray}{c}
\xi_1 + \xi_2 = \xi \\
|\xi_1| \leq n|\xi_0|\,, \\
|\xi_2| \leq m|\xi_0|
\end{subarray}} C(\xi_1, \xi_2) \widehat \vphi_n(t, \xi_1) \widehat \vphi_m(t, \xi_2)
$$
for some constant $C(\xi_1, \xi_2) \in \C$\,. Hence,
$$
\begin{aligned}
|\widehat f(t, \xi)| & \lesssim_{m, n} \sum_{\begin{subarray}{c}
\xi_1 + \xi_2 = \xi \\
|\xi_1| \leq n|\xi_0|\,, \\
|\xi_2| \leq m|\xi_0|
\end{subarray}} |\widehat \vphi_n(t, \xi_1)| |\widehat \vphi_m(t, \xi_2)| \\
& \lesssim_{m, n} C(\vphi_n) C(\vphi_m) \frac{e^{n \lambda_0 t}}{\lambda_0^{n - 1}} \frac{e^{m \lambda_0 t}}{\lambda_0^{m - 1}},
\end{aligned}
$$
from which the thesis follows. 
\end{proof}

\begin{lemma}\label{stime vphi n}
Let $n = 2, \ldots , N$.  There exists a unique  $\vphi_n(t, x) = \sum_{1 \leq |\xi| \leq n|\xi_0|} \widehat \vphi_n(t, \xi) e^{\ii x \cdot \xi}$ satisfying \eqref{eq per vphi 2}, \eqref{eq induttiva per vphi n}. 
Furthermore,
the following estimate holds for some constant $C_n > 0$,
$$
|\widehat \vphi_n(t, \xi)| \leq C_n \frac{e^{n \lambda_0 t}}{\lambda_0^{n - 1}}, \quad \forall t \geq 0, \quad 1 \leq |\xi| \leq n|\xi_0|\,,
$$
and for any $s \geq 0$, 
$$
\| \vphi_n(t, \cdot) \|_{H^s} \lesssim_{s, n}  \frac{e^{n \lambda_0 t}}{\lambda_0^{n - 1}}, \quad \forall t \geq 0\,. 
$$
\end{lemma}

\begin{proof}
We shall prove the lemma by induction on $n \in \{ 2, \ldots, N\}$. 

\medskip

\noindent
{\sc The basis of induction $ n = 2$.} We recall that 
$$
\vphi_1(t, x) = e^{\lambda_0 t} \vphi_{in} (x) , \quad \vphi_{in}(x) = \sum_{|\xi| = |\xi_0|} \widehat \vphi_{in}(\xi) e^{\ii x \cdot \xi}\,. 
$$ 
We seek a solution of the problem:
$$
\begin{cases}
\partial_t \vphi_2 = {\mathcal L}_\omega \vphi_2 + {\mathcal Q}(\vphi_1, \vphi_1)\,, \\
\vphi_2(0, x) = 0.
\end{cases}
$$
We let $f_1 := {\mathcal Q}(\vphi_1, \vphi_1)$. By Lemma \ref{lemma astratto stima forcing}, one has that 
$$
f_1(t, x) = \sum_{1 \leq |\xi| \leq 2|\xi_0|} \widehat f_1(t, \xi) e^{\ii x \cdot \xi},
$$
and
$$
|\widehat f_1(t, \xi) | \lesssim e^{2 \lambda_0 t}, \quad \forall t \geq 0, \quad \forall 1 \leq |\xi| \leq 2|\xi_0|\,. 
$$
Then  by Lemma \ref{lemma variazione costanti En}, we have that 
$$
\vphi_2(t, x) = \sum_{1 \leq |\xi| \leq 2|\xi_0|} \widehat \vphi_2(t, \xi) e^{\ii x \cdot \xi},
$$
with
$$
|\widehat \vphi_2(t, \xi)| \lesssim \frac{e^{2 \lambda_0 t}}{\lambda_0}, \quad \forall t \geq 0, \quad \forall 1 \leq |\xi| \leq 2|\xi_0|,
$$
which is exactly the claim for $n = 2$. 

\medskip

\noindent
{\sc The Induction step.} We now assume that for any $k \in \{ 2, \ldots, n - 1 \}$ with $n \leq N$, the statement of the Lemma holds for $\vphi_2, \vphi_3, \ldots, \vphi_{n - 1}$. We then prove it for $\vphi_{n}$.  From \eqref{eq induttiva per vphi n}, $\vphi_n$ must solve
$$
\begin{cases}
\partial_t \vphi_n = {\mathcal L}_\omega \vphi_n + f_n  \,, \\
 \vphi_n(0, x) = 0,
\end{cases}
$$
where 
$$
f_n := \sum_{k = 1}^{n - 1} {\mathcal Q}(\vphi_{n - k}\,,\, \vphi_k)\,. 
$$
By the induction hypothesis for any $k \in \{ 2, \ldots, n - 1 \}$, one has that 
$$
\begin{aligned}
& \vphi_k(t, \cdot) \in E_k, \quad \forall t \geq 0 \quad \text{and} \\
& |\widehat \vphi_k(t, \xi)| \leq C_k  \frac{e^{k \lambda_0 t}}{\lambda_0^{k - 1}}, \quad \forall t \geq 0, \quad \forall 1 \leq |\xi| \leq k|\xi_0|\,. 
\end{aligned}
$$
Therefore, by applying Lemma \ref{lemma astratto stima forcing}, $f_n(t, \cdot) \in E_n$ for any $t \geq 0$ and 
$$
|\widehat f_n(t, \xi)| \leq C_1(n) \frac{e^{n \lambda_0 t}}{\lambda_0^{n - 2}}, \quad \forall t \geq 0, \quad \forall 1 \leq |\xi| \leq n|\xi_0|\,,
$$
for some constant $C_1(n) \gg 0$ large enough. The hypotheses of Lemma \ref{lemma variazione costanti En} are then verified with 
$$
C(f_n) := \frac{C_1(n)}{\lambda_0^{n - 2}}\,.
$$
We can conclude from this lemma that $\vphi_n(t, \cdot) \in E_n$ satisfies \eqref{eq induttiva per vphi n}, and for any $t \geq 0$, $1 \leq |\xi| \leq n|\xi_0|$,  the following estimate holds:
$$
\begin{aligned}
|\widehat \vphi_n(t, \xi)| & \lesssim_n C(f_n) \frac{e^{n \lambda_0 t}}{\lambda_0} \lesssim_n \frac{e^{n \lambda_0 t}}{\lambda_0^{n - 1}}\,.
\end{aligned}
$$
We conclude that $\vphi_n$ has the required properties and the induction argument is established.
\end{proof}

In order to provide an estimate of the error term $r_N$ in \eqref{errore soluzioni appprossimate}, it will be useful to fix a few parameters.
For any $s > 1$ we fix the time $T_0 \equiv T_0(\e, s) > 0$ in such a way that 
$$
\e \| \vphi_1(T_0, \cdot) \|_{H^s} = 1\,.
$$
By rescaling, we can assume without loss of generality that  $\| \vphi_{in}\|_{H^s} = 1$.
Consequently, we choose $T_0$ so that
\begin{subequations} \label{def tempo T0 T sigma}
\begin{equation}
  \e e^{\lambda_0 T_0} = 1, \quad \text{or equivalently} \quad T_0 := \frac{1}{\lambda_0} \log(\e^{- 1})\,.
\end{equation}
We also choose an intermediate time $T_\sigma$ with
\begin{equation}
 T_\sigma := T_0 - \sigma \quad \text{where} \quad \sigma \gg 0 \quad \text{has to be determined, independently of}\,\,\e\,. 
\end{equation}
\end{subequations}
In order to ensure $T_\sigma \gg 0$ large enough, one has to  assume $T_0 \gg \sigma$. More precisely, we require that 
\begin{equation}\label{T0 T0 2}
T_0 > T_\sigma = T_0 - \sigma >\frac{T_0}{2}
\end{equation}
Therefore, if we choose $\varepsilon$ to satisfy 
\begin{equation}\label{smallness epsilon sigma}
0 < \e < e^{- 2 \lambda_0 \sigma}\,
\end{equation}
then \eqref{T0 T0 2} will be 
satisfied.
This is actually the only condition on $\e$ appearing in the paper. 
We now provide an estimate of the approximate solution $\vphi_{app}.$

\begin{lemma}[\bf Estimate of  $\vphi_{app}$ in \eqref{approx-solution}]\label{lemma stime vphi app}
Let $s > 1$, $N \geq 2$, and let $\vphi_{app}$ be as in \eqref{approx-solution}. There exists a constant $\sigma_0(s, N, \omega) > 0$ such that, if $\sigma \geq \sigma_0(s, N, \omega)$, then the following estimates hold for any $t \in [0, T_\sigma]$:
\begin{equation} \label{eq:lemma stime vphi app.1}
\begin{aligned}
& \| \vphi_{app}(t) - \e \vphi_1(t) \|_{H^s} \leq C(s, N, \omega) \e^2 e^{2 \lambda_0 t}\,, \\
& \| \vphi_{app}(t)  \|_{H^s} \leq 2 \e e^{\lambda_0 t}\,.
\end{aligned}
\end{equation}
Furthermore,
\begin{equation} \label{eq:lemma stime vphi app.2}
\sup_{t \in [0, T_\sigma]} \| \vphi_{app}(t) - \e \vphi_1(t) \|_{H^s} \leq C(s, N, \omega)  e^{- 2 \lambda_0 \sigma}\,. 
\end{equation}
\end{lemma}

\begin{proof}
For any $t \in [0, T_\sigma]$, one has that 
$$
\begin{aligned}
\| \vphi_{app}(t) - \e \vphi_1(t) \|_{H^s} & \leq \sum_{n = 2}^N \e^n \| \vphi_n(t) \|_{H^s} \leq \sum_{n = 2}^N \e^n  \| \vphi_n(t) \|_{H^s} \\
& \lesssim_{s, N, \omega} \sum_{n = 2}^N \e^n e^{n \lambda_0 t} \lesssim_{s, N, \omega} \e^2 e^{2 \lambda_0 t} \sum_{n \geq 2} (\e e^{\lambda_0 t})^{n - 2}\,.
\end{aligned}
$$
We recall that $\e e^{\lambda_0 T_0} = 1$, that $T_\sigma = T_0 - \sigma$, and that $\sigma$, $e$ must satisfy \eqref{T0 T0 2} and \eqref{smallness epsilon sigma}. Therefore,
since  $\e e^{\lambda_0 t} \leq e^{- \lambda_0 \sigma} < \frac12$ for any $t \in [0, T_\sigma]$, we also have that
$$
\| \vphi_{app}(t) - \e \vphi_1(t) \|_{H^s}  \lesssim_{s, N, \omega} \e^2 e^{2 \lambda_0 t} \sum_{n \geq 2} \frac{1}{2^n} \lesssim_{s, N, \omega} \e^2 e^{2 \lambda_0 t}\,.
$$
Consequently,
$$
\sup_{t \in [0, T_\sigma]} \| \vphi_{app}(t) - \e \vphi_1(t) \|_{H^s} \leq C(s, N, \omega) e^{- 2 \lambda_0 \sigma}
$$
for some constant $C(s, N, \omega) > 0$.  Moreover for any $t \in [0, T_\sigma]$, 
$$
\begin{aligned}
\| \vphi_{app}(t) \|_{H^s} & \leq \e \| \vphi_1(t) \|_{H^s} + \| \vphi_{app}(t) - \e \vphi_1(t) \|_{H^s} \\
& \leq \e e^{\lambda_0 t} + C(s, N, \omega) \e^2 e^{2 \lambda_0 t} \leq  \e e^{\lambda_0 t} \Big( 1 + C(s, N, \omega) \e e^{\lambda_0 t} \Big) \\
& \leq 2 \e e^{\lambda_0 t},
\end{aligned}
$$
given that
$$
\begin{aligned}
C(s, N, \omega) \e e^{\lambda_0 t}  & \leq C(s, N, \omega) \e e^{\lambda_0 T_0} e^{- \lambda_0 \sigma} \\
& \stackrel{\e e^{\lambda_0 T_0} = 1}{\leq}C(s, N, \omega) e^{- \lambda_0 \sigma} \leq 1
\end{aligned}
$$
by taking  $\sigma \geq \frac{1}{\lambda_0} \log\Big(C(s, N, \omega)\Big) $.  The proof of the lemma is then complete
\end{proof}

\begin{lemma}[{\bf Estimate of the error term $r_N$ in \eqref{errore soluzioni appprossimate}}] \label{lemma stima rN}
Let $s > 1, N \geq 2$, and let $r_N$ be given as in \eqref{errore soluzioni appprossimate}. Then there exists a constant $\sigma_0(s, N, \omega) > 0$ such that for any $\sigma \geq \sigma_0(s, N, \omega)$ and for any $t \in [0, T_\sigma]$, 
$$
 \| r_N(t) \|_{H^s} \leq C(s, N, \omega) \e^{N + 1} e^{(N + 1)\lambda_0 t}
$$
for some constant $C(s, N, \omega) > 0$. 
\end{lemma}

\begin{proof}
We continue to assume the conditions on $T_0$, $_\sigma$, $\sigma$ and $\e$ as in Lemma \ref{lemma stime vphi app}. 
By \eqref{errore soluzioni appprossimate} for any $t \in [0, T_\sigma]$, 
$$
\begin{aligned}
\| r_N(t) \|_{H^s} & \leq \sum_{n = N + 1}^{2 N} \e^n \sum_{k = 1}^{n - 1} \| {\mathcal Q}(\vphi_{n - k}(t), \vphi_k(t))\|_{H^s} \\
& \lesssim_s \sum_{n = N + 1}^{2 N} \e^n \sum_{k = 1}^{n - 1} \| \vphi_{n - k}(t) \|_{H^{s + 1}}  \| \vphi_k(t))\|_{H^{s + 1}} \\
& \stackrel{Lemma \, \ref{stime vphi n}}{\lesssim_{s, N}} \sum_{n = N + 1}^{2 N} \e^n \sum_{k = 1}^{n - 1}  \frac{e^{n \lambda_0 t}}{\lambda_0^{n - 2}} \lesssim_{s, N} \frac{1}{\lambda_0^{2(N - 1)}}\sum_{n = N + 1}^{2 N} \e^n   e^{n \lambda_0 t} \\
& \lesssim_{s, N, \omega} \sum_{n = N + 1}^{2 N} \e^n   e^{n \lambda_0 t}\,. 
\end{aligned} 
$$
Next we observe that, for any $t \in [0, T_\sigma]$, $n \geq N + 1$,
$$
\begin{aligned}
\e^n e^{n \lambda_0 t} & = \e^{N + 1} e^{(N +1) \lambda_0 t} \e^{n - N - 1} e^{(n - N - 1) \lambda_0 t} \leq \e^{N + 1} e^{(N +1) \lambda_0 t} (\e e^{\lambda_0 T_0} e^{- \lambda_0 \sigma})^{n - N - 1} \\
& \stackrel{\e e^{\lambda_0 T_0} = 1}{\leq} \e^{N + 1} e^{(N +1) \lambda_0 t} e^{- \lambda_0 \sigma (n - N - 1)}\,.
\end{aligned}
$$
Therefore, for any $t \in [0, T_\sigma]$, we have that
$$
\begin{aligned}
\| r_N(t) \|_{H^s} & \stackrel{\eqref{def tempo T0 T sigma}}{\lesssim_{s, N, \omega}} \e^{N + 1} e^{(N + 1) \lambda_0 t} \sum_{n \geq N + 1}  e^{- \sigma \lambda_0(n - N - 1)} \\
& \lesssim_{s, N, \omega} \e^{N + 1} e^{(N + 1) \lambda_0 t}  \sum_{k \geq 0} (e^{- \lambda_0 \sigma})^k \lesssim_{s, N, \omega} \e^{N + 1} e^{(N + 1) \lambda_0 t}
\end{aligned}
$$
since $\sigma \gg 0$ is such that $e^{- \lambda_0 \sigma} < \frac12$. This estimate implies the claim.
\end{proof}


\section{Local in time existence close to the approximate solution}\label{section-local-existence}
We look for a solution by perturbing the approximate solution constructed in the previous section, i.e. 
$$
\vphi = \vphi_{app} + u\,, \quad \text{with} \quad u(0, x) = 0\,. 
$$
Then the perturbation $u$ must satisfy the Cauchy problem 
\begin{equation}\label{PDE app sol u}
\begin{cases}
\partial_t u = {\mathcal L}_\omega u + 2 {\mathcal Q}(\vphi_{app}, u) + {\mathcal Q}(u, u) + r_N, \\
u(0, x) = 0\,. 
\end{cases}
\end{equation}
In this section, we shall prove the following local-in-time existence result. 
We will prove local existence over a time interval $[0, T]$, with $0 < T \ll T_\sigma$. In order to do so, given $T > 0$ we choose $0 < \e \ll 1$ satisfying 
\begin{equation}\label{epsilon local existence}
0 < \e < \e_0(T, \omega), \quad \e_0(T, \omega) := e^{- 2 \lambda_0 T}\,. 
\end{equation}
Then, one has  
\begin{equation}\label{relazione T T0 T sigma}
0 < T < T_0 /2 < T_\sigma < T_0\,.
\end{equation}

\begin{prop}\label{teo esistenza locale}
Let $s \geq 4$, $T \geq 1$, $N \geq 2$ and assume \eqref{epsilon local existence}. Then there exists a constant $\sigma_0(T, s, N, \omega) > 0$ such that if $\sigma \geq \sigma_0(T, s, N, \omega)$, there exists a unique solution 
$$
u \in {\mathcal C}^0 \Big([0, T], H^s_0(\T^2) \Big) \cap {\mathcal C}^1 \Big([0, T], H^{s - 4}_0(\T^2) \Big)
$$
of the Cauchy problem \eqref{PDE app sol u} satisfying the bound 
$$
\| u(t) \|_{H^s} \leq e^{ N \lambda_0 t} \e^N, \quad \forall t \in [0, T]\,.  
$$
\end{prop}

The rest of this section is devoted to the proof of Proposition \ref{teo esistenza locale}.
We seek solutions in so-called ``mild" form, that is, as solutions of the following integral equation:
\begin{equation}\label{eq punto fisso}
u(t) = \int_0^t e^{(t - \tau) {\mathcal L}_\omega} \Big[r_N(\tau) + 2 {\mathcal Q}(\vphi_{app}(\tau), u(\tau)) + {\mathcal Q}(u(\tau), u(\tau))  \Big]\, d \tau \,.
\end{equation}

\begin{rmk}
Because we take initial data in $H^s$, $s \geq 4$, it can be shown that the mild solution is actually a strong solution in ${\mathcal C}^0([0,T],H^s)\cap {\mathcal C}^1([0,T],H^{s-4})$, where by a strong solution we mean that \eqref{KURAMOTO-RESCALED-VARIABLES} is satisfied pointwise in time a.e. in space.
\end{rmk}

For any $T > 0$, we introduce the space  
\begin{equation}\label{palla punto fisso}
\begin{aligned}
& {\mathcal E}_T(s) \equiv {\mathcal E}_T(s, N) := \Big\{ u \in {\mathcal C}^0([0, T], H^s_0(\T^2)) : \| u \|_{s, T} := \sup_{t \in [0, T]} e^{- \lambda_0 N t} \| u(t) \|_{H^s} < + \infty  \Big\}\,. 
\end{aligned}
\end{equation}
We note that, if $u \in {\mathcal E}_T(s, N)$, by definition  the following elementary bound holds:
\begin{equation}\label{norma pesata negative exp}
\| u(t) \|_{H^s} \leq \| u \|_{s, T} e^{N \lambda_0 t} , \qquad \forall t \in [0, T]\,. 
\end{equation}
The operator ${\mathcal L}_\omega$ is  self-adjoint operator and satisfies the hypotheses of the Hille-Yosida Theorem. Therefore,  it generates an exponential semigroup $e^{t {\mathcal L}_\omega}$, $t \geq 0$, on $L^2_0(\mathbb{T}^2)$.
We  provide  some estimates on  $e^{t {\mathcal L}_\omega}$, $t \geq 0$, in the next lemma. 

\begin{lemma}\label{lemma stime semigruppo libero}
Let $\vphi \in H^s_0(\T^2)$, $s \geq0$. Then 
\begin{enumerate}[label=(\roman*),ref=({\em \roman*)}]
\item $\| e^{t {\mathcal L}_\omega} \vphi \|_{H^s} \leq e^{\lambda_0 t} \| \vphi \|_{H^s}\,.$ 
\label{i:lemma stime semigruppo libero.1}
\item Moreover, 
$\| e^{t {\mathcal L}_\omega} \vphi \|_{H^{s + 1}} \lesssim_\omega \Big( e^{\lambda_0 t} + \dfrac{1}{t^{\frac14}} \Big) \| \vphi \|_{H^s}\,.$
\label{i:lemma stime semigruppo libero.2}
\end{enumerate}
\end{lemma}

\begin{proof}
\ {\sc Proof of \ref{i:lemma stime semigruppo libero.1}.} \ 
Since $\lambda_0$ is the maximum (positive) eigenvalue of ${\mathcal L}_\omega$, using that $\lambda_\omega(\xi) \leq \lambda_0 $ for any $\xi \in \Z^2 \setminus \{ 0 \}$ one has that for any $s \geq 0$, $t \geq 0$,  
$$
\begin{aligned}
\| e^{t {\mathcal L}_\omega t}\vphi \|_{H^s} & = \Big( \sum_{\xi \in \Z^2 \setminus \{ 0 \}} e^{ 2 \lambda_\omega(\xi) t} \langle \xi \rangle^{2 s} |\widehat \vphi(\xi)|^2\Big)^{\frac12} \leq e^{ \lambda_0 t} \Big( \sum_{\xi \in \Z^2 \setminus \{ 0 \}}  \langle \xi \rangle^{2 s} |\widehat \vphi(\xi)|^2\Big)^{\frac12} \leq e^{ \lambda_0 t} \| \vphi \|_{H^s}\,.
\end{aligned}
$$


\noindent {\sc Proof of\ref{i:lemma stime semigruppo libero.2}.} \ We  note that, for any $\xi \in \Z^2 \setminus \{ 0 \}$, $|\xi| > \frac{\sqrt{2}}{\omega}$, 
$$
\omega^2 |\xi|^2 - 1 \geq \frac12 \omega^2 |\xi|^2.
$$
Therefore,
\begin{equation}\label{parabolico quarto ordine}
\lambda_\omega(\xi) = - \omega^2 |\xi|^2 (\omega^2|\xi|^2- 1) \leq - \frac{\omega^4}{2} |\xi|^4 \,. 
\end{equation}
We then split 
$$
\begin{aligned}
& e^{t {\mathcal L}_\omega} \vphi (x) = u_1(t, x) + u_2(t, x)\,, \\
& u_1(t, x) := \sum_{1 \leq |\xi| \leq \frac{\sqrt{2}}{\omega} } e^{\lambda_\omega(\xi) t} \widehat \vphi(\xi) e^{\ii x \cdot \xi}\,, \quad u_2(t, x) :=  \sum_{|\xi| > \frac{\sqrt{2}}{\omega}} e^{ \lambda_\omega(\xi) t} \widehat \vphi(\xi) e^{\ii x \cdot \xi} \,.
\end{aligned}
$$
We first estimate $u_1$ in $H^s$. We have
$$
\begin{aligned}
\| u_1 (t, \cdot )\|_{H^s} & = \Big( \sum_{1 \leq |\xi| \leq \frac{\sqrt{2}}{\omega} } e^{2\lambda_\omega(\xi) t} \langle \xi \rangle^{2(s + 1)}|\widehat \vphi(\xi)|^2  \Big)^{\frac12} \\
& \leq e^{\lambda_0 t}\Big( \sum_{1 \leq |\xi| \leq \frac{\sqrt{2}}{\omega} }  \langle \xi \rangle^{2(s + 1)}|\widehat \vphi(\xi)|^2  \Big)^{\frac12} \\
& \lesssim_\omega e^{\lambda_0 t}  \Big( \sum_{1 \leq |\xi| \leq \frac{\sqrt{2}}{\omega} }  \langle \xi \rangle^{2s }|\widehat \vphi(\xi)|^2  \Big)^{\frac12} \lesssim_\omega e^{\lambda_0 t} \| \vphi \|_{H^s}\,,
\end{aligned}
$$
where we have used again that $\lambda_\omega(\xi) \leq \lambda_0$ for any $\xi \in \Z^2 \setminus \{ 0 \}$).
Next we estimate $\| u_2(t, \cdot) \|_{H^{s + 1}}$:
$$
\begin{aligned}
\| u_2(t, \cdot) \|_{H^{s + 1}}  & = \Big( \sum_{|\xi| > \frac{\sqrt{2}}{\omega} } e^{2 t \lambda_\omega (\xi)} |\xi |^{2(s + 1)} |\widehat \vphi(\xi)|^2 \Big)^{\frac12} \\
& \stackrel{\eqref{parabolico quarto ordine}}{\leq} \sup_{|\xi| > \frac{\sqrt{2}}{\omega}} \Big(  | \xi | e^{ - \frac{\omega^4}{2} |\xi|^4 t} \Big) \| \vphi \|_{H^s} \lesssim_\omega \frac{1}{t^{\frac14}} \| \vphi \|_{H
^s}\,. 
\end{aligned}
$$
The claim then follows by combining the estimates on $u_1$ and $u_2$. 
%
%
\end{proof}

The next results concern bounds on the integral terms in \eqref{eq punto fisso}.
First,  we tackle the remainder term.

\begin{lemma}\label{stima propagatore con rN}
For any $t \in [0, T_\sigma]$, the following estimate holds for any $s > 1$
$$
\Big\| \int_0^t e^{(t - \tau) {\mathcal L}_\omega} r_N(\tau)\, d \tau \Big\|_{H^s} \leq C(s, N, \omega) \e^{N + 1} e^{(N + 1) \lambda_0 t}\,. 
$$
As a consequence, 
$$
\sup_{t \in [0, T_\sigma]} e^{- N \lambda_0 t} \Big\| \int_0^t e^{(t - \tau) {\mathcal L}_\omega} r_N(\tau)\, d \tau \Big\|_{H^s} \leq C(s, N, \omega) \e^N e^{- \lambda_0 \sigma}\,. 
$$
\end{lemma}

\begin{proof}
Using Lemma \ref{lemma stima rN}, one has that for any $t \in [0, T_\sigma]$, 
$$
\begin{aligned}
\Big\| \int_0^t e^{(t - \tau) {\mathcal L}_\omega} [r_N(\tau)]\, d \tau \Big\|_{H^s}  & \leq \int_0^t \Big\| e^{(t - \tau) {\mathcal L}_\omega} [r_N(\tau)] \Big\|_{H^s}\, d \tau \\
& \lesssim \int_0^t e^{\lambda_0 (t - \tau)} \| r_N(\tau) \|_{H^s}\, d \tau \lesssim_{s, N, \omega} e^{\lambda_0 t} \e^{N + 1} \int_0^t e^{- \lambda_0 \tau} e^{ (N + 1)\lambda_0 \tau}\, d \tau \\
& \lesssim_{s, N, \omega} e^{\lambda_0 t} \e^{N + 1} \int_0^t e^{ N \lambda_0 \tau}\, d \tau \lesssim_{s, N, \omega} \e^{N + 1} e^{(N + 1) \lambda_0 t}\,.
\end{aligned}
$$
The second estimate of the lemma  follows using that for any $t \in [0, T_\sigma]$,
$$\e e^{\lambda_0 t} \leq \e e^{\lambda_0 T_\sigma} =  \e e^{\lambda_0 T_0} e^{- \lambda_0 \sigma} = e^{- \lambda_0 \sigma},$$ 
since $\e e^{\lambda_0 T_0} = 1$. 
\end{proof}

In the next lemma, we bound the time-integral of the semigroup applied to 
the non-linearity $\mathcal{Q}$ for various entries.  

\begin{lemma}\label{lemma stime parte quadratica}
Let $s > 2$, $ T \geq 1$, $N \geq 2$. Then there exists a constant $\sigma_0(T, s, N, \omega) > 0$ such that for any $\sigma \geq \sigma_0(T, s, N, \omega)$, for any $\varepsilon$ satisfying  \eqref{smallness epsilon sigma}, \eqref{epsilon local existence} then the following holds. 
\begin{enumerate}[label=(\roman*),ref={\em (\roman*)}]
 \item If $u \in {\mathcal E}_T(s, R)$,  then for any $t \in [0, T]$
$$
\Big\|  \int_{0}^t e^{(t - \tau) {\mathcal L}_\omega}[{\mathcal Q}(\vphi_{app}(\tau), u(\tau))] \, d \tau \Big\|_{s, T} \leq C(s, N, \omega) T^{\frac34} e^{- \lambda_0 \sigma } \| u \|_{s, T}.
$$
\label{i:lemma stime parte quadratica.1}
\item  If $u, v \in {\mathcal E}_T(s, N)$, then 
$$
\Big\| \int_{0}^t e^{(t - \tau) {\mathcal L}_\omega}[{\mathcal Q}(v(\tau), u(\tau))] \, d \tau\Big\|_{s, T} \leq C(s, N, \omega) T^{\frac34} e^{N \lambda_0 T} \| u \|_{s, T} \| v \|_{s, T}.
$$
\label{i:lemma stime parte quadratica.2}
\end{enumerate}
\end{lemma}

\begin{proof}
{\sc Proof of \ref{i:lemma stime parte quadratica.1}.}\  We recall that by Lemma \ref{lemma stime vphi app} one has that 
$$
\| \vphi_{app}(t) \|_{H^s} \leq 2 \e e^{\lambda_0 t}, \quad \forall t \in [0, T_\sigma]\,. 
$$
We also use that since $s - 1 > 1$, $H^s_0(\T^2)$  is an algebra and, hence,  for any $\vphi, u \in H^s_0(\T^2)$ one has that 
$$
\| {\mathcal Q}(u, \vphi) \|_{H^{s - 1}} \lesssim_s \| u \|_{H^s} \| \vphi \|_{H^s}\,. 
$$
Then for any $t \in [0, T]$, one also has:
\begin{multline}
 \Big\| \int_{0}^t e^{(t - \tau) {\mathcal L}_\omega}[{\mathcal Q}(\vphi_{app}(\tau), u(\tau))] \, d \tau\Big\|_{H^{s}}    \leq  \int_{0}^t \Big\| e^{(t - \tau) {\mathcal L}_\omega}[{\mathcal Q}(\vphi_{app}(\tau), u(\tau))] \Big\|_{H^{s }} \, d \tau \\
  \stackrel{\text{Lemma}\, \ref{lemma stime semigruppo libero}}{\lesssim_\omega} \int_0^t e^{\lambda_0 (t - \tau)} \| {\mathcal Q}(\vphi_{app}(\tau), u(\tau)) \|_{H^{s - 1}}\, d \tau  
  +  \int_0^t \dfrac{1}{(t - \tau)^{\frac14}}\| {\mathcal Q}(\vphi_{app}(\tau), u(\tau)) \|_{H^{s - 1}}\, d \tau \\
  \lesssim_{\omega, s} \int_0^t e^{\lambda_0 (t - \tau)} \| \vphi_{app}(\tau) \|_{H^s} \| u(\tau) \|_{H^s}\, d \tau \, d \tau
  +  \int_0^t \dfrac{1}{(t - \tau)^{\frac14}} \| \vphi_{app}(\tau) \|_{H^s} \| u(\tau) \|_{H^s} \, d \tau \\
  \lesssim_{\omega, s} \int_0^t e^{\lambda_0 (t - \tau)} \e e^{\lambda_0 \tau} \| u(\tau) \|_{H^s}\, d \tau \, d \tau  
  +  \int_0^t \dfrac{1}{(t - \tau)^{\frac14}}  \e e^{\lambda_0 \tau}\| u(\tau) \|_{H^s} \, d \tau \\
  \lesssim_{\omega, s} \e e^{\lambda_0 t} \int_0^t  \| u(\tau) \|_{H^s}\, d \tau   
   + \e e^{\lambda_0 t} \int_0^t \dfrac{1}{(t - \tau)^{\frac14}}  \| u(\tau) \|_{H^s} \, d \tau \\
  \stackrel{\eqref{norma pesata negative exp}}{\lesssim_{s, \omega}}\e e^{\lambda_0 t} \int_0^t  e^{\lambda_0 N \tau}\, d \tau  \| u \|_{s, T} 
   + \e e^{\lambda_0(N + 1) t} \int_0^t \dfrac{1}{(t - \tau)^{\frac14}}   \, d \tau \| u\|_{s, T} \\
  \lesssim_{s, \omega} \e e^{\lambda_0 t}  e^{\lambda_0 N t}T^{\frac34} \| u \|_{s, T}.
\end{multline}
We observe that since $t \leq T \leq T_0/2 < T_\sigma < T_0$, one has that 
$$
\e e^{\lambda_0 t} \leq e^{- \lambda_0 \sigma}
$$
and hence the claim follows. 

\medskip

\noindent {\sc Proof of \ref{i:lemma stime parte quadratica.2}.}\ 
 For any $t \in [0, T]$, one has that 
\begin{multline}
 \Big\| \int_{0}^t e^{(t - \tau) {\mathcal L}_\omega}[{\mathcal Q}( u(\tau), v(\tau))] \, d \tau\Big\|_{H^{s}}    \leq  \int_{0}^t \Big\| e^{(t - \tau) {\mathcal L}_\omega}[{\mathcal Q}(u(\tau), v(\tau))] \Big\|_{H^{s }} \, d \tau \\
  \stackrel{\text{Lemma}\, \ref{lemma stime semigruppo libero}}{\lesssim_\omega} \int_0^t e^{\lambda_0 (t - \tau)} \| {\mathcal Q}(u(\tau), v(\tau)) \|_{H^{s - 1}}\, d \tau  
  +  \int_0^t \dfrac{1}{(t - \tau)^{\frac14}}\| {\mathcal Q}(u(\tau), v(\tau)) \|_{H^{s - 1}}\, d \tau \\
  \lesssim_{\omega, s} \int_0^t e^{\lambda_0 (t - \tau)} e^{2 N \lambda_0 \tau}\, d \tau  \| u \|_{s, T} \| v \|_{s, T}  
  +  \int_0^t \dfrac{e^{2 N \lambda_0 \tau}}{(t - \tau)^{\frac14}}\, d \tau \| u \|_{s, T} \| v \|_{s, T} \\
  \lesssim_{\omega, s} e^{2N \lambda_0 t} T^{\frac34} \| u \|_{s, T} \| v \|_{s, T}\,. 
\end{multline}
As a consequence,
$$
\begin{aligned}
\Big\| \int_{0}^t e^{(t - \tau) {\mathcal L}_\omega}[{\mathcal Q}( u(\tau), v(\tau))] \, d \tau \Big\|_{s, T} & = \sup_{t \in [0, T]} e^{- \lambda_0 N t} \Big\| \int_{0}^t e^{(t - \tau) {\mathcal L}_\omega}[{\mathcal Q}( u(\tau), v(\tau))] \, d \tau \Big\|_{H^s} \\
& \lesssim_{\omega, s} e^{N \lambda_0 T} T^{\frac34} \| u \|_{s, T} \| v \|_{s, T},
\end{aligned}
$$
which is the claimed bound. 
\end{proof}

We now  let
\begin{equation}\label{scelta raggio contr locale}
R := 3 C(s, N, \omega) \e^N e^{- \lambda_0 \sigma},
\end{equation}
and consider the ball 
$$
{\mathcal B}_T(s, R) := \Big\{ u \in {\mathcal E}_T(s) : \| u \|_{s, T} \leq R\Big\}.
$$
Solutions to \eqref{eq punto fisso}  are fixed points of the map
\begin{equation}\label{def mappa Phi punto fisso}
\Phi(u)(t) := \int_0^t e^{(t - \tau) {\mathcal L}_\omega} \Big[r_N(\tau) + 2 {\mathcal Q}(\vphi_{app}(\tau), u(\tau)) + {\mathcal Q}(u(\tau), u(\tau))  \Big]\, d \tau.
\end{equation}
We next prove that this map is contractive,  from which one can deduce immediately Proposition \ref{teo esistenza locale}. 

\begin{lemma}\label{lemma finale punto fisso Phi}  
Let $s > 2$, $ T \geq 1$, $N \geq 2$. 
Then there exists a constant $\sigma_0(T, s, N, \omega) > 0$ such that for any 
$\sigma$ satisfying $\sigma \geq \sigma_0(T, s, N, \omega)$, for any $\varepsilon$ satisfying  \eqref{smallness epsilon sigma}, \eqref{epsilon local existence}, the map $\Phi : {\mathcal B}_T(s, R) \to {\mathcal B}_T(s, R)$ is a contraction.
\end{lemma}

\begin{proof}
We first establish that $\Phi$ maps the ball to itself.
Let $\| u \|_{s, T} \leq R$, then by Lemmata \ref{stima propagatore con rN}, \ref{lemma stime parte quadratica}, one has that 
$$
\| \Phi(u) \|_{s, T} \leq C(s, N, \omega) e^{- \lambda_0 \sigma} \e^N + C(s, N, \omega) T^{\frac34} e^{- \lambda_0 \sigma } R + C(s, N, \omega) T^{\frac34} e^{N \lambda_0 T}  R^{2}.
$$
We then conclude $\|\Phi(u)\|_{s,T}\leq R$ provided that
$$
C(s, N, \omega) e^{- \lambda_0 \sigma} \e^N \leq \frac{R}{3}\,, \quad C(s, N, \omega) T^{\frac34} e^{- \lambda_0 \sigma }  \leq \frac13\, \quad \mathrm{and}\quad C(s, N, \omega) T^{\frac34} e^{N \lambda_0 T} R \leq \frac13\,.
$$
The first condition is satisfied by \eqref{scelta raggio contr locale}. The second condition is satisfied by taking 
$$
\sigma \geq \frac{1}{\lambda_0} \log \Big( 3C(s, N, \omega) T^{\frac34} \Big) \,. 
$$
By the definition of $R$ in \eqref{scelta raggio contr locale}, the third condition is equivalent to 
$$
3 (C(s, N, \omega))^2 T^{\frac34} (e^{ \lambda_0 T}  \e)^N e^{- \lambda_0 \sigma} \leq \frac13.
$$
Since $T \leq T_\sigma = T_0 - \sigma \leq T_0$ and $\e e^{\lambda_0 T} \leq \e e^{\lambda_0 T_0} = 1$, 
this cobstraint will be satisfied if
$$
3 (C(s, N, \omega))^2 T^{\frac34} (e^{ \lambda_0 T}  \e)^N e^{- \lambda_0 \sigma} \leq 3 (C(s, N, \omega))^2 T^{\frac34}  e^{- \lambda_0 \sigma}  \leq \frac13,
$$
which this in turn holds by taking 
$$
\sigma \geq \frac{1}{\lambda_0} \log\Big(9 (C(s, N, \omega))^2 T^{\frac34} \Big)\,. 
$$

We next establish the contraction property.
Let $u_1, u_2$ be such that $\| u_1 \|_{s, T}, \| u_2 \|_{s, T} \leq R$. Then for any $t \in [0, T]$, 
$$
\begin{aligned}
\Phi(u_1)(t) - \Phi(u_2)(t) & = \int_0^t e^{(t - \tau) {\mathcal L}_\omega} \Big[ 2 {\mathcal Q}(\vphi_{app}(\tau), u_1(\tau) - u_2(\tau)) + {\mathcal Q}(u_1(\tau), u_1(\tau)) - {\mathcal Q}(u_2(\tau), u_2(\tau)) \Big]\, d \tau \\
& = \int_0^t e^{(t - \tau) {\mathcal L}_\omega} \Big[ 2 {\mathcal Q}(\vphi_{app}(\tau), u_1(\tau) - u_2(\tau)) \Big]\, d \tau  \\
& \quad + \int_0^t e^{(t - \tau) {\mathcal L}_\omega} \Big[  {\mathcal Q}(u_1(\tau), u_1(\tau) - u_2(\tau)) \Big]\, d \tau  \\
& \quad +   \int_0^t e^{(t - \tau) {\mathcal L}_\omega} \Big[  {\mathcal Q}(u_1(\tau) - u_2(\tau),  u_2(\tau)) \Big]\, d \tau\,.
\end{aligned}
$$
It follows from Lemma \ref{lemma stime parte quadratica} that
$$
\begin{aligned}
 \| \Phi(u_1 ) - \Phi(u_2) \|_{s, T} & \leq C(s, N, \omega) T^{\frac34} e^{- \lambda_0 \sigma } \| u_1 - u_2 \|_{s, T}  \\
 & \quad + C(s, N, \omega) T^{\frac34} e^{N \lambda_0 T} \big(\| u_1 \|_{s, T} + \| u_2 \|_{s, T} \big) \| u_1 - u_2 \|_{s, T} \\ 
 & \leq \Big(C(s, N, \omega) T^{\frac34} e^{- \lambda_0 \sigma } + 2 C(s, N, \omega) T^{\frac34} e^{N \lambda_0 T} R  \Big) \| u_1 - u_2 \|_{s, T} \\
 & \leq \frac12 \| u_1 - u_2 \|_{s, T}\,,  
\end{aligned}
$$
provided that 
$$
\begin{aligned}
& C(s, N, \omega) T^{\frac34} e^{- \lambda_0 \sigma } \leq \frac14\,,  \\
& 2 C(s, N, \omega) T^{\frac34} e^{N \lambda_0 T} R \stackrel{\eqref{scelta raggio contr locale}}{=} 6 C(s, N, \omega)^2 T^{\frac34} (e^{ \lambda_0 T} \e)^N e^{- \lambda_0 \sigma} \leq \frac14\,.
\end{aligned}
$$
The first condition above is fulfilled by taking 
$$
\sigma \geq \frac{1}{\lambda_0} \log\Big(4 C(s, N, \omega) T^{\frac34} \Big)\,.
$$
The second condition can be verified as follows. Since $0 <T < T_0 /2 < T_\sigma = T_0 - \sigma < T_0$ and $e^{\lambda_0 T}\e \leq e^{\lambda_0 T_0} \e = 1$, the second condition is equivalent to 
$$
6 C(s, N, \omega)^2 T^{\frac34} (e^{ \lambda_0 T} \e)^N e^{- \lambda_0 \sigma} \leq 6 C(s, N, \omega)^2 T^{\frac34}  e^{- \lambda_0 \sigma} \leq \frac14,
$$
which holds provided
$$
\sigma \geq \frac{1}{\lambda_0} \log\Big(24C(s, N, \omega)^2 T^{\frac34} \Big)\,.
$$
Hence $\Phi : {\mathcal B}_T(s, R) \to {\mathcal B}_T(s, R)$ is a contraction and the proof of the lemma is concluded. 
\end{proof}

\section{Continuation argument, long time existence and proof of Theorem \ref{main theorem statement}}\label{sezione long time existence}
In this section we fix the Sobolev index 
\begin{equation}\label{indice sobolev soglia minima finale}
s \geq 4\,. 
\end{equation}
We  define 
\begin{equation}\label{maximal time}
\begin{aligned}
T_* & := \sup\Big\{ T \geq 1  : u \in {\mathcal C}^0([0, T], H^s_0(\T^2)) \cap {\mathcal C}^1([0, T], H^{s - 4}_0(\T^2)) \quad \text{solves} \quad \eqref{PDE app sol u}  \\
& \qquad \text{and}\qquad \| u \|_{s, T}\leq  \e^N \Big\}\,. 
\end{aligned}
\end{equation}
We want to prove that $T_* > T_\sigma$. We argue by contradiction. Let us assume that $T_* \leq T_\sigma$. By the definition of $T_*$,  one immediately has that $u \in {\mathcal C}^0\big( [0, T_*], H^s_0(\T^2) \big)$ and 
\begin{equation}\label{stima u s T*}
\begin{aligned}
& \| u \|_{s, T_*} = \sup_{t \in [0, T_*]} e^{- N \lambda_0 t} \| u(t) \|_{H^s} = \e^N\,. 
\end{aligned}
\end{equation} 
We shall perform an energy estimate on the Cauchy problem 
\begin{equation}\label{prob cauchy per stima di energia}
\begin{cases}
\partial_t u = {\mathcal L}_\omega u + 2 {\mathcal Q}(\vphi_{app}, u) + {\mathcal Q}(u,  u) + r_N, \\
u(0, x) = 0\,.
\end{cases}
\end{equation}
We will use the Bony paralinearization formula (see \eqref{paraprodotto}) to control the quadratic terms. This formula reads 
$$
\nabla v \cdot \nabla \vphi = T_{\nabla \vphi} \cdot \nabla v + T_{\nabla v} \cdot \nabla \vphi + {\mathcal R}_B(\nabla v ,  \nabla \vphi),
$$
where 
$$
\begin{aligned}
T_{\nabla \vphi} \cdot \nabla v & = \sum_{k = 1}^2  T_{\partial_{x_k} \vphi} \cdot \partial_{x_k} v\,, \quad T_{\nabla v} \cdot \nabla \vphi & =  \sum_{k = 1}^2  T_{\partial_{x_k} v}  \partial_{x_k} \vphi\,, \\
{\mathcal R}_B(\nabla v ,  \nabla \vphi) & := \sum_{k = 1}^2 {\mathcal R}_B(\partial_{x_k} v ,  \partial_{x_k} \vphi)\,.
\end{aligned}
$$
The paraproduct expansion gives
$$
\begin{aligned}
& 2 {\mathcal Q}(\vphi_{app}(t), u(t)) = 2 T_{\nabla \vphi_{app}(t)} \cdot \nabla u(t) + 2 T_{\nabla u(t)} \cdot \nabla \vphi_{app}(t) + 2 {\mathcal R}_B(\nabla \vphi_{app}(t), \nabla u(t))\,, \\
&  {\mathcal Q}(u(t), u(t)) = 2 T_{\nabla u(t)} \cdot \nabla u(t) + {\mathcal R}_B(\nabla u(t) , \nabla u(t)).
\end{aligned}
$$
Therefore,
\begin{equation}\label{def fN t}
\begin{aligned}
& 2 {\mathcal Q}(\vphi_{app}(t), u(t)) + {\mathcal Q}(u(t),  u(t)) = T_{A(t)} \cdot \nabla u(t) + f_N(t)\,, \\
& A(t) := 2 \big( \nabla \vphi_{app}(t) + \nabla u(t)\big)\,, \\
& f_N(t) := r_N(t)  + 2 T_{\nabla u(t)} \cdot \nabla \vphi_{app}(t) + 2 {\mathcal R}_B(\nabla \vphi_{app}(t), \nabla u(t)) + {\mathcal R}_B(\nabla u(t), \nabla u(t))\,. 
\end{aligned}
\end{equation}
We then rewrite  \eqref{prob cauchy per stima di energia} as 
\begin{equation}\label{prob cauchy per stima di energia 1}
\begin{cases}
\partial_t u = {\mathcal L}_\omega u + T_{A} \cdot \nabla u + f_N(t),  \\
u(0, x) = 0\,. 
\end{cases}
\end{equation}
From Lemma \ref{lemma stime vphi app}, by \eqref{stima u s T*} and  using that $T_* \leq  T_\sigma$, it follows thar
\begin{equation}\label{stime vphi app stime u}
\|\vphi_{app}(t) \|_{H^{s + 1}} \leq 2 \e e^{\lambda_0 t}, \quad \| u (t) \|_{H^s} \leq e^{N \lambda_0 t} \e^N, \quad \forall t \in [0, T_*],
\end{equation}
which implies 
\begin{equation}\label{stima trasporto paralinearizzato}
\begin{aligned}
\| A(t) \|_{H^{s - 1}} & \leq 2 (\| \vphi_{app}(t) \|_{H^s} + \| u(t) \|_{H^s}) \leq 4\e e^{\lambda_0 t} + 2\e^N e^{N \lambda_0 t} 
\\
& \leq \e e^{\lambda_0 t}(4 + 2 (\e e^{\lambda_0 t})^{N - 1})\leq  6 \e e^{\lambda_0 t}.
\end{aligned}
\end{equation}
Above, we have used again that $N \geq 1$ and that $0 \leq t \leq T_*\leq T_\sigma = T_0 - \sigma$. Consequently, $\e e^{\lambda_0 t} \leq \e e^{\lambda_0 T_0} e^{- \lambda_0 \sigma} \leq e^{- \lambda_0 \sigma} \leq 1$ for $\sigma >0$ large enough.

Next, Lemma \ref{continuit paradiff} (applied with $d = 2$, $s_0 = 2$) implies that for any $t \in [0, T_*]$ 
\begin{equation}\label{stima T nabla u nabla vphi app}
\begin{aligned}
\| T_{\nabla u(t)} \cdot \nabla \vphi_{app}(t) \|_{H^s} & \lesssim_s \| \nabla u(t) \|_{H^2} \| \nabla \vphi_{app}(t) \|_{H^s} \lesssim_s \| u(t) \|_{H^3} \| \vphi_{app}(t) \|_{H^{s + 1}} \\
&  \stackrel{\eqref{indice sobolev soglia minima finale}}{\lesssim_s} \| u(t) \|_{H^s} \| \vphi_{app}(t) \|_{H^{s + 1}} \stackrel{\eqref{stime vphi app stime u}}{\lesssim_s} e^{(N + 1) \lambda_0 t} \e^{N + 1}\,,
\end{aligned}
\end{equation}
while Lemma \ref{stima astratta resto paraproduct} (applied with $d = 2$, $s_0 = 2$, $s_1 = 1$, $s_2 = s - 1$) implies that for any $t \in [0, T_*]$ 
\begin{equation}\label{stima RB nabla vphi app nabla u}
\begin{aligned}
\| {\mathcal R}_B(\nabla \vphi_{app}(t), \nabla u(t)) \|_{H^s} & \lesssim_s  \| \nabla \vphi_{app} (t)\|_{H^{3}} \| \nabla u(t)\|_{H^{s - 1}}  \lesssim_s \| \vphi_{app}(t) \|_{H^4} \| u(t) \|_{H^s} \\
&  \stackrel{\eqref{indice sobolev soglia minima finale}}{\lesssim_s} \| \vphi_{app}(t) \|_{H^{s + 1}} \| u(t) \|_{H^s} \stackrel{\eqref{stime vphi app stime u}}{\lesssim_s } e^{(N + 1) \lambda_0 t} \e^{N + 1},
\end{aligned}
\end{equation}
and 
\begin{equation}\label{stima RB nabla u nabla u}
\begin{aligned}
\| {\mathcal R}_B(\nabla u(t), \nabla u(t)) \|_{H^s} & \lesssim_s  \| \nabla u (t)\|_{H^{3}} \| \nabla u(t)\|_{H^{s - 1}}  \lesssim_s \| u(t) \|_{H^4} \| u(t) \|_{H^s} \\
&  \stackrel{\eqref{indice sobolev soglia minima finale}}{\lesssim_s} \| u(t) \|_{H^s}^2 \stackrel{\eqref{stime vphi app stime u}}{\lesssim_s } \Big(e^{N \lambda_0 t} \e^{N} \Big)^2  \lesssim_s e^{(N + 1) \lambda_0 t} \e^{N + 1}.
\end{aligned}
\end{equation}
Above we have used that 
$$
\Big(e^{N \lambda_0 t} \e^{N} \Big)^2 = \Big( e^{(N + 1) \lambda_0 t} \e^{N + 1} \Big) \Big( e^{(N - 1) \lambda_0 t} \e^{N - 1} \Big),
$$
again since  $N \geq 2$ and $0 \leq t \leq T_* \leq T_\sigma = T_0 - \sigma \leq T_0$.
Furthermore,
$$
e^{(N - 1) \lambda_0 t} \e^{N - 1} = \Big( e^{\lambda_0 t} \e\Big)^{N - 1} {\leq} \Big( e^{\lambda_0 T_0} \e\Big)^{N - 1} = 1\,. 
$$
Hence from Lemma \ref{lemma stima rN}, by  estimates \eqref{stima T nabla u nabla vphi app}, \eqref{stima RB nabla vphi app nabla u}, \eqref{stima RB nabla u nabla u} and recalling the definition of $f_N$ in \eqref{def fN t}, we obtain
\begin{equation}\label{stime forzante nuova0}
\| f_N(t) \|_{H^s} \lesssim_{s,N, \omega} \e^{N + 1} e^{(N + 1) \lambda_0 t}, \qquad \forall t \in [0, T_*]\,. 
\end{equation}

%
%
%
In what follows, we will employ the notation
$$
\| u \|_{H^s} := \| \Lambda^s u \|_{L^2} \quad \text{where} \quad \Lambda^s u(x) = \sum_{\xi \in \Z^2 \setminus \{ 0 \}} \langle \xi \rangle^s \widehat u(\xi) e^{\ii x \cdot \xi}\,. 
$$
In order to perform energy estimates in a rigorous way, we regularize the Cauchy problem \eqref{prob cauchy per stima di energia 1}. To this end, we introduce a cut-off function $\chi$ satisfying  
\begin{equation}\label{cut off bla.1}
\begin{aligned}
& \chi \in {\mathcal C}^\infty_c(\R^d), \quad 0 \leq \chi \leq 1, \\
& \chi(\xi) = 1, \quad \forall |\xi| \leq 1\,, \\
& \chi(\xi) = 0, \quad \forall |\xi| \geq 2.
\end{aligned}
\end{equation}
For any $M > 0$ we set 
\begin{equation}\label{def chi M}
\chi_M(\xi) := \chi(M^{- 1} \xi), \quad \xi \in \R^d.
\end{equation}
The following elementary properties hold: 
\begin{equation}\label{cut off bla.2}
\begin{aligned}
& \chi_M \in {\mathcal C}^\infty_c(\R^d), \quad 0 \leq \chi_M \leq 1, \\
& \chi_M(\xi) = 1, \quad \forall |\xi| \leq M\,, \\
&  \chi_(\xi) = 0, \quad \forall |\xi| \geq 2M\,, \\
& \partial_\xi^\alpha \chi_M(\xi) = 0, \quad \text{if} \quad \alpha \neq 0, \quad |\xi| \leq M \quad \text{or} \quad |\xi| \geq 2M \,,\\
& |\partial_\xi^\alpha \chi_M(\xi)| \lesssim_\alpha \langle \xi \rangle^{- |\alpha|}, \quad \forall \xi \in \R^d\, \quad \text{uniformly with respect to} \quad M \geq  0. 
\end{aligned}
\end{equation}

We next define the mollifier 
\begin{equation}\label{def cal SK}
{\mathcal S}_M : L^2(\T^d) \to {\mathcal C}^\infty(\T^d), \quad u(x) \mapsto {\mathcal S}_M u(x) := \sum_{\xi \in \Z^d } \chi_M(\xi) \widehat u(\xi) e^{\ii x \cdot \xi}\,. 
\end{equation} 
We then consider the regularized Cauchy problem by setting $u_M(t) := {\mathcal S}_M u(t)$, $t \in [0, T_*]$. Then by \eqref{prob cauchy per stima di energia 1} $u_M(t)$ solves 
\begin{equation}\label{regularized cauchy problem}
\begin{cases}
\partial_t u_M = {\mathcal L}_\omega u_M +  T_{A} \cdot \nabla u_M + f_{N, M}(t),  \\
u_M(0, x) = 0\,,
\end{cases}
\end{equation}
where
\begin{equation}\label{def f N M}
f_{N, M}(t) := {\mathcal S}_M f_N(t) + [{\mathcal S}_M, T_{A} \cdot \nabla] u(t)\,.
\end{equation}
Lemma \ref{commutatore mollifier} (applied with $d = 2$, $s_0 = 2$), the estimates on $f_N$ in \eqref{stime forzante nuova0}, and the fact that $\| {\mathcal S}_M u \|_{H^s} \leq \| u \|_{H^s}$ give
\begin{equation}\label{stime forzante nuova}
\begin{aligned}
\| f_{N, M}(t) \|_{H^s} & \leq \| f_N(t) \|_{H^s} + \| [{\mathcal S}_M, T_{A} \cdot \nabla] u(t)\|_{H^s} \\
& \lesssim_{s, N,  \omega} \e^{N + 1} e^{(N + 1) \lambda_0 t} + \| A(t) \|_{H^3} \| u(t) \|_{H^s} \\
& \stackrel{\eqref{stime vphi app stime u}, \eqref{stima trasporto paralinearizzato}}{\lesssim_{s,N,  \omega}} \e^{N + 1} e^{(N + 1) \lambda_0 t}\,.  
\end{aligned}
\end{equation}
 We then have 
\begin{align}
\frac{d}{d t} \| u_M(t) \|_{H^s}^2 & = \big\langle \Lambda^s \partial_t u_M(t) \,,\, \Lambda^s u_M(t) \big\rangle_{L^2} + \big\langle \Lambda^s u_M(t) \,,\, \Lambda^s  \partial_t  u_M(t) \big\rangle_{L^2}  \nonumber \\
& = \big\langle {\mathcal L}_\omega \Lambda^s u_M(t)\,,\, \Lambda^s u_M(t) \big\rangle_{L^2} + \big\langle \Lambda^s u_M(t)\,,\, {\mathcal L}_\omega \Lambda^s u_M(t) \big\rangle_{L^2} \nonumber\\
& \quad + \big\langle \Lambda^s (T_{A(t)} \cdot \nabla u_m(t))\,,\, \Lambda^s u_M(t)\big\rangle_{L^2} +  \big\langle \Lambda^s u_M(t)\,,\, \Lambda^s (T_{A(t)} \cdot \nabla u(t)) \big\rangle_{L^2} \nonumber\\
& \quad + \big\langle \Lambda^s {\mathcal S}_Mf_{N, M}(t)\,,\, \Lambda^s u_M(t) \big\rangle_{L^2} + \big\langle \Lambda^s u_M(t)\,,\, \Lambda^s {\mathcal S}_Mf_{N, M}(t) \big\rangle_{L^2} \nonumber\\
& \leq  2 \big\langle {\mathcal L}_\omega \Lambda^s u_M(t)\,,\, \Lambda^s u_M(t) \big\rangle_{L^2} + \big\langle T_{A(t)} \cdot \nabla \Lambda^s u_M(t)\,,\, \Lambda^s u_M(t) \big\rangle_{L^2}  \label{stima di energia 0}\\
& \quad + \big\langle \Lambda^s u_M(t)\,,\,  T_{A(t)} \cdot \nabla \Lambda^s u_M(t) \big\rangle_{L^2}   + \big\langle [\Lambda^s\,,\, T_{A(t)} \cdot \nabla] u_M(t)\,,\, \Lambda^s u_M(t) \big\rangle_{L^2} \nonumber\\
& \quad + \big\langle \Lambda^s u_M(t)\,,\,  [\Lambda^s\,,\, T_{A(t)} \cdot \nabla] u_M(t)  \big\rangle_{L^2} + 2 \| f_{N, M}(t) \|_{H^s} \| u(t) \|_{H^s} \nonumber\\
& \leq 2 \big\langle {\mathcal L}_\omega \Lambda^s u_M(t)\,,\, \Lambda^s u_M(t) \big\rangle_{L^2} + \Big\langle \Big(T_{A(t)} \cdot \nabla + (T_{A(t)} \cdot \nabla)^* \Big) \Lambda^s u_M(t)\,,\, \Lambda^s u_M(t) \Big\rangle_{L^2}  \nonumber\\ 
& \quad  + 2 \|  [\Lambda^s\,,\, T_{A(t)} \cdot \nabla] u_M(t) \|_{L^2} \| u_M(t) \|_{H^s}  + 2 \| f_{N, M}(t) \|_{H^s} \| u_M(t) \|_{H^s}\,. \nonumber
\end{align}

We estimate each term above separately. 
We observe that ${\mathcal L}_\omega = {\rm diag}_{\xi \in \Z^2 \setminus \{ 0 \}} \lambda_\omega(\xi)$ and that $\lambda_\omega(\xi) \leq \lambda_0$ for any $\xi \in \Z^2 \setminus \{ 0 \}$. Consequently,
\begin{equation}\label{stima forma quadratica cal L omega}
\begin{aligned}
\big\langle {\mathcal L}_\omega \Lambda^s u_M\,,\, \Lambda^s u_M \big\rangle_{L^2} & = \sum_{\xi \in \Z^2 \setminus \{ 0 \}} \lambda_\omega(\xi) \langle \xi \rangle^{2 s} |\widehat u_M(\xi)|^2 \\
& \leq \lambda_0 \sum_{\xi \in \Z^2 \setminus \{ 0 \}}  \langle \xi \rangle^{2 s} |\widehat u_M(\xi)|^2  \leq \lambda_0 \| u_M \|_{H^s}^2\,. 
\end{aligned}
\end{equation}
We also note that $(T_A \cdot \nabla)^* = \sum_{k = 1}^2 (T_{A_k} \partial_{x_k})^*$. Thus, by Lemmata \ref{lemma commutatore Lambda s}, \ref{lemma aggiunto Ta dx} (applied with $d = 2, s_0 = 2$) and by the estimate \eqref{stima trasporto paralinearizzato},  for any $t \in [0, T_*]$
\begin{equation}\label{stime paraprodotti nell energia}
\begin{aligned}
& \|  [\Lambda^s\,,\, T_{A(t)} \cdot \nabla] u_M \|_{L^2}\,,\,\Big\| \Big(T_A \cdot \nabla + (T_A \cdot \nabla)^* \Big) \Lambda^s u_M \Big\|_{L^2}   \\
& \lesssim_s \| A(t) \|_{H^3} \| u_M(t) \|_{H^s} \stackrel{\eqref{stima trasporto paralinearizzato}, \eqref{indice sobolev soglia minima finale}, \eqref{stima u s T*}}{\lesssim_s}   \e^{N + 1} e^{(N + 1) \lambda_0 t}\,. 
\end{aligned}
\end{equation}
The two latter estimates together with \eqref{stime forzante nuova}, \eqref{stima forma quadratica cal L omega} and \eqref{stima di energia 0} imply that for any for any $t \in [0, T_*]$,
$$
\frac{d}{d t} \| u_M(t) \|_{H^s}^2 \leq  \lambda_0  \| u_M(t) \|_{H^s}^2 +   C_0 \e^{2 N + 1} e^{(2 N + 1) \lambda_0 t},
$$
for some constant $C_0 \equiv C_0(s, N, \omega) > 0$. 
Since $u_M(0) = 0$, by the comparison principle for ODEs one has that 
$$
\| u_M(t) \|_{H^s}^2 \leq z(t), \quad \forall t \in [0, T_*],
$$
where $z(t)$ solves the problem
$$
\begin{cases}
\dot z(t) = \lambda_0 z(t) +  C_0 \e^{2 N + 1} e^{(2 N + 1) \lambda_0 t} \\
z(0) = 0. 
\end{cases}
$$
Hence,
$$
\begin{aligned}
z(t) & = C_0  \e^{2 N + 1} \int_0^t e^{\lambda_0(t - \tau)}  e^{(2 N + 1) \lambda_0 t}\, d \tau   \stackrel{\eqref{stime forzante nuova}}{\lesssim_{s,N,  \omega}} \e^{2 N + 1} e^{\lambda_0 t}  \int_0^t  e^{2 N \lambda_0 \tau}\, d \tau  \\
& \lesssim_{s,N, \omega}  \e^{2(N + 1)} e^{2(N + 1) \lambda_0 t}.
\end{aligned}
$$
We then have that
\begin{equation}\label{final estimate uM}
\| u_M(t) \|_{H^s} \leq C(s,N, \omega) \e^{N + \frac12} e^{(N + \frac12) \lambda_0 t}, \quad \forall t \in [0, T_*]\, , 
\end{equation}
for some constant $C(s,N, \omega) >  0$. Since
$$
\lim_{M \to + \infty}\| {\mathcal S}_M u - u \|_{H^s} = 0, \quad \forall u \in H^s(\T^2),
$$ 
and $u(t) \in H^s(\T^2)$ for any $t \in [0, T_*]$, passing to the limit in the inequality \eqref{final estimate uM} gives
$$
\| u(t) \|_{H^s} \leq C(s,N, \omega) \e^{N + \frac12} e^{(N + \frac12) \lambda_0 t}, \quad \forall t \in [0, T_*]\,. 
$$
Now, since $T_* \leq T_\sigma = T_0 - \sigma$ and $\e e^{\lambda_0 T_0} = 1$, by choosing $\sigma \gg 0$ so large that
$$
C(s,N, \omega) (\e e^{\lambda_0 t})^{\frac12} \leq C(s,N, \omega) (\e e^{\lambda_0 T_0})^{\frac12} e^{- \frac{\lambda_0 \sigma}{2}} \leq C(s,N, \omega) e^{- \frac{\lambda_0 \sigma}{2}} \leq \frac{1}{2},
$$
we have
$$
\| u \|_{s, T_*} = \sup_{t \in [0, T_*]} e^{- N \lambda_0 t} \| u(t) \|_{H^s} \leq \frac12 \e^N\,.
$$
This latter inequality contradicts \eqref{stima u s T*}.

\subsection{Conclusion of the proof of Theorem \ref{main theorem statement}}\label{sezione finale dim teorema}
From the contraction property for the map $\Phi$, the only solution $\vphi$ of the Cauchy problem \eqref{Prob Cauchy instab} is given by $\vphi = \vphi_{app} + u$,  is an element of the space
$$
{\mathcal C}^0([0, T_\sigma], H^s_0(\T^2)) \cap {\mathcal C}^1([0, T_\sigma], H^{s - 4}_0(\T^2)),
$$
and satisfies the bound
$$
\| u (t) \|_{H^s} \leq \e^N e^{N \lambda_0 t}, \quad \forall t \in [0, T_\sigma].
$$

Hence, using again that $\e e^{\lambda_0 t} \leq e^{- \lambda_0 \sigma}$ for any $t \in [0, T_\sigma]$, one obtains that 
$$
\sup_{t \in [0, T_\sigma]} \| u(t) \|_{H^s} \leq e^{- N \lambda_0 \sigma}\,.
$$
We can fix $N := 2$. Therefore, the above estimate and \eqref{eq:lemma stime vphi app.1} imply for any $t \in [0, T_\sigma]$, 
$$
\begin{aligned}
\sup_{t \in [0, T_\sigma]} \| \vphi(t) - \e \vphi_1(t) \|_{H^s} & \leq \sup_{t \in [0, T_\sigma]} \| \vphi_{app}(t) - \e \vphi_1(t) \|_{H^s} + \sup_{t \in [0, T_\sigma]} \| u(t) \|_{H^s} \\
& \leq C(s, N, \omega) e^{- 2 \lambda_0 \sigma} + e^{- N \lambda_0 \sigma} \stackrel{N = 2}{\leq} C_1(s, N, \omega) e^{- 2 \lambda_0 \sigma},
\end{aligned}
$$
where $C_1(s, N, \omega) := 1 + C(s, N, \omega)$. Moreover, since $\vphi_1(t, x) = e^{\lambda_0 t} \vphi_{in }(x)$ where $\| \vphi_{in } \|_{H^s} = 1$, it follows that
$$
\e \| \vphi_1(T_\sigma) \|_{H^s} = \e e^{\lambda_0 T_\sigma} \| \vphi_{in} \|_{H^s} = \e e^{\lambda_0 T_0} e^{- \lambda_0 \sigma} \stackrel{\eqref{def tempo T0 T sigma}}{=} e^{- \lambda_0 \sigma}\,.
$$
Hence, one gets 
$$
\begin{aligned}
\| \vphi(T_\sigma) \|_{H^s} & \geq \e \| \vphi_1(T_\sigma) \|_{H^s} - \sup_{t \in [0, T_\sigma]} \| \vphi(t) - \e \vphi_1(t) \|_{H^s}  \\
& \geq  e^{- \lambda_0 \sigma} - C_1(s, N, \omega) e^{- 2 \lambda_0 \sigma} \geq \frac12 e^{- \lambda_0 \sigma},
\end{aligned}
$$
provided that 
$$
\sigma \geq \frac{1}{\lambda_0} \log(2 C_1(s, N, \omega))\,. 
$$
Finally we define the constants $\delta_0(s, \omega), \e_0(s, \omega)$ appearing in the Theorem \ref{main theorem statement} as
$$
\delta_0(s, \omega) := \frac12 e^{- \lambda_0 \sigma}, \quad \e_0(s, \omega) := e^{- 2 \lambda_0 \sigma}\,. 
$$
The proof of Theorem \ref{main theorem statement} is complete.

\appendix
\section{Some tools from Para-differential calculus}\label{appendice paradiff}
In this appendix we shall collect some properties of para-differential operators. We also give all the proofs in the quantitative way needed in the rest of the paper. For an exhaustive theory of para-differential operators we refer to the monographs \cite{Metivier}, \cite{Taylor}. 

\noindent
Let $\psi \in C^\infty_c(\R)$ be such that 
\begin{equation}\label{cut-off-1}
\begin{aligned}
& 0 \leq \psi \leq 1,  \\
& \psi (y) = 1, \quad \forall y \in [-1/2,1/2]\,, \\
&   \psi(y) = 0 \,,\quad \forall y \in \R \setminus [- 1, 1]\,. 
\end{aligned}
\end{equation}
Given $0 < \delta < 1$, we  let 
\begin{equation}\label{cut-off-2}
\psi_\delta(\eta, \xi) := \psi \Big( \frac{\langle \eta \rangle}{\delta \langle \xi \rangle} \Big), \quad (\eta, \xi) \in \R^d \times \R^d\,. 
\end{equation}
We then define the paraproduct operator $T_a$ for given a Sobolev functions $a\in H^{s_0}(\T^d)$, $s_0 > \frac{d}{2}$, as
\begin{equation}\label{def paraproduct}
T_a u(x)  : = \sum_{\xi \in \Z^d} \sigma_a(x, \xi) \widehat u(\xi) e^{\ii x \cdot \xi}\,,\quad \sigma_a(x, \xi) := \sum_{\eta \in \Z^d} \psi_\delta(\eta, \xi) \widehat a(\eta) e^{\ii x \cdot \eta}\,.
\end{equation}
In the sequel, we do not emphasize the dependence of all constants on $\delta$ since it has to be taken small enough but it is a fixed parameter. The following lemma holds. 
\begin{lemma}\label{continuit paradiff}
Let $a \in H^{s_0}(\T^d)$, $s_0 > d/2$. Then the operator $T_a$ is a bounded linear operator from $H^{s}(\T^d)$ into itself for any $s \geq 0$ and 
\begin{equation}\label{stima Ta}
\| T_a \|_{{\mathcal B}(H^{s})} \lesssim_{s} \| a \|_{H^{s_0}}\,.
\end{equation}
\end{lemma}
\begin{proof}
By \eqref{cut-off-1}, \eqref{cut-off-2} one has that $\widehat{\sigma_a}(\eta, \xi) = \psi_\delta(\eta, \xi) \widehat a(\eta)$ satisfies  $${\rm supp}(\widehat \sigma_a) \subseteq \{ (\eta, \xi) \in \Z^d \times \R^d : \langle \eta \rangle\leq \delta \langle \xi \rangle \}. $$ 
Moreover, 
$$
\widehat{T_a u}(\eta)   = \sum_{\begin{subarray}{c}
\xi \in \Z^d 
\end{subarray}} \psi_\delta(\eta - \xi, \xi) \widehat a(\eta - \xi) \widehat u(\xi) = \sum_{\begin{subarray}{c}
\xi \in \Z^d  \\
\langle \eta - \xi \rangle \leq \delta \langle \xi \rangle
\end{subarray}} \psi_\delta(\eta - \xi, \xi) \widehat a(\eta - \xi) \widehat u(\xi) \,.
$$
We then estimate
\begin{eqnarray}
\| T_a u\|_{H^s}^2 & \leq &  \sum_{\eta \in \Z^d} \Big(  \sum_{\begin{subarray}{c}
\xi \in \Z^d \\
\langle  \eta - \xi \rangle \leq \delta \langle \xi \rangle
\end{subarray}} \langle \eta \rangle^{ s} |\psi_\delta(\eta - \xi, \xi)|  |\widehat a(\eta - \xi)| | \widehat u(\xi)| \Big)^2\,. \label{Ta u s} 
\end{eqnarray}
Since $\langle \xi - \eta \rangle \leq \delta \langle \xi \rangle$, it holds
$$
\langle \eta \rangle^s \lesssim_s  \langle \xi \rangle^s + \langle \eta - \xi \rangle^s \lesssim_s  \langle \xi \rangle^s\,. 
$$
Therefore, using the Cauchy-Schwartz inequality and the fact that $\sum_{\xi \in \Z^d} \langle \eta - \xi \rangle^{2 s_0} = \sum_{k \in \Z^d} \langle k \rangle^{- 2 s_0} = C(s_0) < + \infty$, as $s_0 > d/2$, we conclude that
\begin{align*}
\eqref{Ta u s} & \lesssim_s  \sum_{\eta \in \Z^d} \Big( \sum_{\xi \in \Z^d} \langle \xi \rangle^s |\widehat a(\eta - \xi)| |\widehat u (\xi)| \Big)^2 \nonumber\\
& \lesssim_s  \sum_{\eta \in \Z^d} \Big( \sum_{\xi \in \Z^d} \langle \eta - \xi \rangle^{s_0} |\widehat a(\eta- \xi)| \langle \xi \rangle^s |\widehat u (\xi )| \frac{1}{\langle \eta - \xi \rangle^{s_0}} \Big)^2  \nonumber\\
& \lesssim_s  \sum_{\xi, \eta \in \Z^d} \langle \eta - \xi \rangle^{2 s_0} |\widehat a(\eta - \xi)|^2 \langle \xi \rangle^{2 s} |\widehat u(\xi)|^2 \nonumber\\
& \lesssim_s \sum_{\xi \in \Z^d} \langle \xi \rangle^{2 s} |\widehat u (\xi)|^2 \sum_{\eta \in \Z^d} \langle \eta - \xi \rangle^{2 s_0} |\widehat a(\eta - \xi)|^2  \nonumber\\
& \lesssim_s \| a \|_{H^{s_0}} \| u \|_{H^s}\,. 
\end{align*}
The latter estimate then implies \eqref{stima Ta}.
\end{proof}

Given two Sobolev functions $a, u \in H^s$, $s > s_0 > d/2$, we can split their product in the following way: 
\begin{equation}\label{paraprodotto}
a u = T_a u + T_u a + {\mathcal R}_B(a, u)\,, 
\end{equation}
where the remainder ${\mathcal R}_B(a, u)$ is defined as follows:
\begin{equation}\label{resto paraprodotto}
{\mathcal R}_B(a, u) := \sum_{\eta, \xi \in \Z^d} \omega(\eta, \xi)\widehat a(\eta) \widehat u(\xi) e^{\ii x \cdot (\eta + \xi)}\,,\quad \omega(\eta, \xi) := 1- \psi_\delta(\eta, \xi)- \psi_\delta(\xi, \eta)\,.
\end{equation}
We remark that that for $0 < \delta < 1$ sufficiently small, there are two positive constants $0 < C_1 < C_2$ such that 
\begin{equation}\label{supporto di omega}
{\rm Supp}(\omega) \subseteq \big\{ (\eta, \xi) \in \R^{2 d} :  C_1 \langle \xi \rangle \leq \langle \eta \rangle \leq C_2 \langle \xi \rangle  \big\}\,.
\end{equation}
In the following theorem,  we  estimate of the remainder ${\mathcal R}_B(a, u)$ in Sobolev spaces, showing that it is a regularizing operator both in $a$ and $u$.
\begin{lemma}\label{stima astratta resto paraproduct}
Let $s_1, s_2 \geq 0$, $s_0 > d/2$, $a \in H^{s_0 + s_1}(\T^d)$, $u \in H^{s_2}(\T^d)$. The remainder ${\mathcal R}_B(a, u)$ defined in \eqref{resto paraprodotto} satisfies the following estimate: 
\begin{equation}\label{stima resto paraprodotto}
\|{\mathcal R}_B(a, u) \|_{H^{s_1 + s_2}} \lesssim_{s_1, s_2} \| a \|_{H^{s_1 + s_0}} \| u \|_{H^{s_2}}\,.
\end{equation}
\end{lemma}
\begin{proof}
Let $f := {\mathcal R}(a, u)$. We have 
\begin{eqnarray}
\widehat f(\xi) & =  \sum_{\eta \in \Z^d} \omega(\xi - \eta, \eta) \widehat a(\xi - \eta) \widehat u(\eta)\,.
\end{eqnarray}
As a consequence,
\begin{eqnarray}
\| f \|_{H^{s_1 + s_2}}^2 & = & \sum_{\xi \in \Z^d} \langle \xi \rangle^{2(s_1 + s_2)} |\widehat f(\xi)|^2 \leq \sum_{\xi \in \Z^d} \Big( \sum_{\eta \in \Z^d} \langle \xi \rangle^{s_1 + s_2} |\omega(\xi - \eta, \eta)| | \widehat a(\xi - \eta)|| \widehat u(\eta)| \Big)^2 \nonumber\\
& \stackrel{\eqref{supporto di omega}}{\leq} & \sum_{\xi \in \Z^d} \Big( \sum_{ C_1 \langle \eta \rangle \leq \langle \xi - \eta\rangle \leq C_2 \langle \eta \rangle} \langle \xi \rangle^{s_1 + s_2} | \widehat a(\xi - \eta)|| \widehat u(\eta)| \Big)^2\,. \label{marcolino}
\end{eqnarray}
We notice that 
$$
\begin{aligned}
& \langle \xi \rangle \lesssim \langle \eta \rangle + \langle \eta - \xi \rangle \lesssim \langle \eta - \xi \rangle\,, \\
&  \langle \xi \rangle  \lesssim \langle \eta \rangle + \langle \eta - \xi \rangle \lesssim \langle \eta \rangle\,.
\end{aligned}
$$
Therefore, 
$$
\langle \xi \rangle^{s_1 + s_2} \lesssim_{s_1, s_2} \langle \xi - \eta \rangle^{s_1}\langle \eta\rangle^{s_2}.
$$
Thus, by using the Cauchy-Schwartz inequality, it follows that 
\begin{eqnarray}
\eqref{marcolino} & \lesssim_{s_1, s_2} &  \sum_{\xi, \eta \in \Z^d} \langle \xi - \eta \rangle^{2(s_1 + s_0)} \widehat a(\xi - \eta) \langle \eta \rangle^{2 s_2} \widehat u(\eta) \lesssim_{s_1, s_2}  \| a \|_{H^{s_1 + s_0}} \| u \|_{H^{s_2}}\,, \nonumber
\end{eqnarray}
which implies estimate \eqref{stima resto paraprodotto}.
\end{proof}

\begin{lemma}\label{lemma commutatore Lambda s}
Let $a \in H^{s_0 + 1}(\T^d)$, $s_0 > \frac{d}{2} $, $k \in \{ 1, \ldots, d \}$. Then for any $s \geq 0$, $u \in H^s(\T^d)$, one has that 
$$
\| [ \Lambda^s, T_a \partial_{x_k}] u \|_{L^2} \lesssim_s \| a \|_{H^{s_0 + 1}} \| u \|_{H^s}\,. 
$$
\end{lemma}

\begin{proof}
Let $A := [\Lambda^s, T_a \partial_{x_k} ]$. A direct calculation shows that
$$
\widehat{A u}(\eta) =  \sum_{\xi \in \Z^d} g (\eta, \xi) \widehat a(\eta - \xi) \widehat u(\xi),
$$
where 
$$
g (\eta, \xi) :=  \ii \big( \langle \eta \rangle^s - \langle  \xi \rangle^s\big)  \xi_k \psi_\delta(\eta - \xi\,,\, \xi).
$$
Because of the cut-off function $\psi_\delta$, one has that 
$$
{\rm supp}( g ) \subseteq \Big\{ (\eta, \xi) \in \R^d \times \R^d : \langle \eta - \xi \rangle \leq \delta \langle \xi \rangle \Big\}\,.
$$
Moreover, by the mean value theorem 
$$
|\langle \eta \rangle^s - \langle \xi \rangle^s| \lesssim_s (\langle \eta \rangle^{s - 1} + \langle \xi \rangle^{s - 1})|\eta - \xi|.
$$
Next, on the support of $m$ we have
$$
\langle \eta \rangle \lesssim \langle \eta - \xi\rangle + \langle \xi \rangle \lesssim \langle \xi \rangle,
$$
so that
$$
|g  (\eta, \xi)| \lesssim \langle \xi \rangle^s \langle \eta - \xi \rangle\,.
$$
Since  $2s_0  > d$ and hence $\sum_{\xi} \langle \eta - \xi \rangle^{- 2 s_0 } = \sum_{k} \langle k \rangle^{- 2s_0 } = C(s_0) < + \infty$, the Cauchy-Schwartz inequality implies
$$
\begin{aligned}
\| A u \|_{L^2}^2 & \lesssim_s \sum_{\eta \in \Z^d} \Big( \sum_{\xi \in \Z^d}  \langle \xi \rangle^s \langle \eta - \xi \rangle^{s_0 + 1} |\widehat a(\eta - \xi)| |\widehat u(\xi)| \frac{1}{\langle \eta - \xi \rangle^{s_0 }} \Big)^2 \\
& \lesssim_s  \sum_{\eta, \xi \in \Z^d} \langle \xi \rangle^{2 s} |\widehat u(\xi)|^2 \langle \eta - \xi \rangle^{2 (s_0 + 1)} |\widehat a(\eta - \xi)|^2 \lesssim_s \| a \|_{H^{s_0 + 1}}^2 \| u \|_{H^s}^2, 
\end{aligned}
$$
which establishes the claimed estimate.
\end{proof}

\begin{lemma}\label{lemma adjoint paraproduct}
Let $\rho > 0$, $s_0 > d/2$, $a \in H^{s_0 + \rho}(\T^d)$,  be real valued. Then for any $s \geq 0$, $u \in H^s(\T^d)$, one has 
$$
\| (T_a^* - T_a) u \|_{H^{s + \rho}} \lesssim_{s, \rho} \| a \|_{H^{s_0 + \rho}} \| u \|_{H^s}\,. 
$$
\end{lemma}

\begin{proof}
A direct calculation shows that 
$$
T_a^* u(x) = \sum_{\eta \in \Z^d} \Big( \sum_{\xi \in \Z^d} \overline{\widehat \sigma_a(\xi - \eta, \eta)} \widehat u(\xi)  \Big)e^{\ii \eta \cdot x} = \sum_{\eta \in \Z^d} \Big( \sum_{\xi \in \Z^d}  \psi_\delta(\xi - \eta, \eta) \overline{\widehat a(\xi - \eta)} \widehat u(\xi)  \Big)e^{\ii \eta \cdot x}
$$
Then, by the definition of $\psi_\delta$ and using that $a$ is real valued ($\overline{\widehat a(\xi - \eta)} = \widehat a(\eta - \xi)$), we have that
$$
T_a^* u(x) =  \sum_{\eta \in \Z^d} \Big( \sum_{\xi \in \Z^d}  \psi_\delta(\eta - \xi, \eta) \widehat a(\eta - \xi) \widehat u(\xi)  \Big)e^{\ii \eta \cdot x}\,.
$$
Thus ${\mathcal R} := T_a^* - T_a$ is 
$$
{\mathcal R}u(x) = \sum_{\eta \in \Z^d} \sum_{\xi \in\Z^d} g (\eta, \xi) \widehat a(\eta - \xi) \widehat u(\xi)  e^{\ii x \cdot \eta}\,,
$$
where
$$
g (\eta, \xi) := \psi_\delta(\eta - \xi, \eta) - \psi_\delta(\eta - \xi, \xi)\,.
$$
and 
$$
|g (\eta, \xi)| \leq 2 \quad \forall (\eta, \xi) \in \R^d \times \R^d\,. 
$$
We next study the support of $g $. By the properties of the function $\psi_\delta$, one has that 
$$
{\rm supp}(g ) \subseteq \Big\{ (\eta, \xi) \in \R^d \times \R^d : \frac{\delta}{2} {\rm min}(\langle \eta \rangle, \langle \xi \rangle) \leq  \langle \eta - \xi \rangle \leq \delta {\rm max}(\langle \eta \rangle, \langle \xi \rangle) \Big\}\,. 
$$
Then if $(\eta, \xi) \in {\rm supp}( g ) $, the following estimates hold for a positive constant $C_0 > 0$:
$$
\begin{aligned}
& \langle \eta \rangle \lesssim \langle \xi \rangle + \langle \eta - \xi \rangle \lesssim  \langle \xi \rangle + \delta {\rm max}(\langle \eta \rangle, \langle \xi \rangle) \leq C_0(1 + \delta)  \langle \xi \rangle + C_0 \delta \langle \eta \rangle, \\
& \langle \xi \rangle \lesssim \langle \eta \rangle + \langle \xi - \eta \rangle \lesssim \langle \eta \rangle + \delta {\rm max}(\langle \eta \rangle, \langle \xi \rangle) \leq C_0 (1 + \delta) \langle \eta \rangle + C_0 \delta \langle \xi \rangle.
\end{aligned} 
$$
 Hence, by choosing $0 < \delta \ll 1$ so small that $1 - C_0 \delta > 0$, we can conclude that
$$
\langle \xi \rangle \leq K_1 \langle \eta \rangle, \quad \langle \eta \rangle \leq K_1 \langle \xi \rangle, \quad K_1 := \frac{C_0(1 + \delta)}{1 - C_0 \delta}.
$$
The above inequality then also implies  
$$
\langle \eta \rangle \lesssim {\rm max}(\langle \eta \rangle , \langle \xi \rangle) \lesssim {\rm min}(\langle \eta \rangle , \langle \xi \rangle) \lesssim \langle \eta - \xi \rangle\,.
$$
Consequently,
\begin{equation}\label{prop supporto m}
{\rm supp}( g ) \subseteq \Big\{ (\eta, \xi) \in \R^d \times \R^d : \langle \eta \rangle \lesssim \langle \xi \rangle, \quad \langle \eta \rangle \lesssim  \langle \eta - \xi \rangle  \Big\}\,. 
\end{equation}
Now, by \eqref{prop supporto m} for any $(\eta, \xi) \in {\rm supp}( g )$ it holds
$$
\langle \eta \rangle^{s + \rho} = \langle \eta \rangle^s \langle \eta \rangle^\rho \lesssim_{s, \rho} \langle \xi \rangle^s \langle \eta - \xi \rangle^\rho.
$$
Therefore,  the Cauchy-Schwartz inequality, using that $s_0 > \frac{d}{2}$ and $\sum_{\xi }\langle \eta - \xi \rangle^{- 2 s_0} = \sum_{k} \langle k \rangle^{- 2 s_0} = C(s_0) < + \infty$, gives
$$
\begin{aligned}
\| {\mathcal R} u \|_{H^{s + \rho}}^2 & \leq \sum_{\eta \in \Z^d} \Big( \sum_{\xi \in \Z^d} \langle \eta \rangle^{s + \rho} | g (\eta, \xi)| |\widehat a(\eta- \xi)| |\widehat u(\xi)|\Big)^2 \\
& \lesssim_{s, \rho}  \sum_{\eta \in \Z^d} \Big( \sum_{\xi \in \Z^d} \langle \eta - \xi \rangle^{\rho + s_0} | g (\eta, \xi)| |\widehat a(\eta- \xi)|  \langle \xi \rangle^s|\widehat u(\xi)| \dfrac{1}{\langle \eta - \xi \rangle^{s_0}}\Big)^2 \\
& \lesssim_{s, \rho} \sum_{\xi \in \Z^d} \langle \xi \rangle^{2 s} |\widehat u(\xi)|^2 \sum_{\eta \in \Z^d} \langle \eta - \xi \rangle^{2(s_0 + \rho)} |\widehat a(\eta - \xi)|^2 \lesssim \| u \|_{H^s}^2 \| a \|_{H^{s_0 + \rho}}^2\,,
\end{aligned}
$$
which implies the claimed bound. 
\end{proof}

From the last lemma we can deduce the following result.

\begin{lemma}\label{lemma aggiunto Ta dx}
Let $k \in \{ 1, \ldots, d \}$, $s_0 > d/2$, $a \in H^{s_0 + 1}(\T^d)$,  be real valued. Then for any $s \geq 0$, $u \in H^s(\T^d)$, one has 
$$
\Big\| \Big[ T_a \partial_{x_k} + (T_a \partial_{x_k})^* \Big] u \Big\|_{H^{s }} \lesssim_s \| a \|_{H^{s_0 + 1}} \| u \|_{H^s}\,. 
$$
\end{lemma}

\begin{proof}
We set  ${\mathcal R}_a := T_{a}^* - T_a $. Then 
$$
\begin{aligned}
(T_a \partial_{x_k})^* + T_a \partial_{x_k} & = - \partial_{x_k} T_a^* + T_a \partial_{x_k} = - \partial_{x_k} T_a + T_a \partial_{x_k} - \partial_{x_k} {\mathcal R}_a \\
& = - T_a \partial_{x_k} + T_a \partial_{x_k} - T_{\partial_{x_k} a} - \partial_{x_k} {\mathcal R}_a =  - T_{\partial_{x_k} a} - \partial_{x_k} {\mathcal R}_a \,.
\end{aligned}
$$
By Lemma \ref{continuit paradiff}, since $a \in H^{s_0 + 1}$, $s_0 > \frac{d}{2}$, one has that for any $s \geq 0$, 
$$
\| T_{\partial_{x_k} a} \|_{{\mathcal B}(H^{s})} \lesssim_s \| \partial_{x_k} a \|_{H^{s_0 }} \lesssim_s \| a \|_{H^{s_0 + 1}}\,.
$$
Moreover, applying Lemma \ref{lemma adjoint paraproduct} with $\rho = 1$ gives for any $s \geq 0$, $u \in H^s(\T^d)$, 
$$
\begin{aligned}
\| \partial_{x_k} {\mathcal R}_a u \|_{H^s} & \leq \| {\mathcal R}_a u \|_{H^{s + 1}} \lesssim \| a \|_{H^{s_0 + 1}} \| u \|_{H^s}\,,
\end{aligned}
$$
which concludes the proof.
\end{proof}

\begin{lemma}\label{commutatore mollifier}
Let ${\mathcal S}_M$ be the mollifier defined in \eqref{def cal SK} and  let $s_0 > \frac{d}{2}$, $a \in H^{s_0 + 1}(\T^d)$. Then for any $s \geq 0$, one has that $\| [{\mathcal S}_M, T_a \partial_{x_k} ] \|_{{\mathcal B}(H^s)} \lesssim_s \| a \|_{H^{s_0 + 1}}$. 
\end{lemma}

\begin{proof}
Let $A := [{\mathcal S}_M, T_a \partial_{x_k} ]$. A direct calculation shows that
$$
\widehat{A u}(\eta) =  \sum_{\xi \in \Z^d} g (\eta, \xi) \widehat a(\eta - \xi) \widehat u(\xi),
$$
where 
$$
g (\eta, \xi) :=  \ii \big( \chi_M(\eta) - \chi_M(\xi)\big)  \xi_k \psi_\delta(\eta - \xi\,,\, \xi)\,.
$$
By the properties of the cut-off function $\psi_\delta$, one has that 
$$
{\rm supp}(g ) \subseteq \Big\{ (\eta, \xi) \in \R^d \times \R^d : \langle \eta - \xi \rangle \leq \delta \langle \xi \rangle \Big\}\,.
$$
Moreover, by the properties of the cut off function $\chi_M$ in \eqref{cut off bla.1}-\eqref{cut off bla.2} and by the mean value theorem 
$$
|\chi_M(\eta) - \chi_M(\xi)| \lesssim \langle \xi + \theta (\eta - \xi) \rangle^{- 1} |\eta - \xi| 
$$
for some $\theta \in [0, 1]$. Since $\langle \eta - \xi \rangle  \leq \delta \langle \xi \rangle$, by choosing $\delta$ small enough, we have that $\langle \xi + \theta (\eta - \xi) \rangle \gtrsim \langle \xi \rangle$. Therefore, 
$$
|\chi_M(\eta) - \chi_M(\xi)| \lesssim \langle \xi \rangle^{- 1} \langle \eta - \xi \rangle\,.
$$
This last bound actually implies that 
$$
|g (\eta, \xi)| \lesssim \langle \eta - \xi \rangle\,. 
$$
Next,  on the support of $m$ it holds
$$
\langle \eta \rangle \lesssim \langle \eta - \xi\rangle + \langle \xi \rangle \lesssim \langle \xi \rangle\,
$$
By the Cauchy-Schwartz inequality, using that $2s_0  > d$ and hence $\sum_{\xi} \langle \eta - \xi \rangle^{- 2 s_0 } = \sum_{k} \langle k \rangle^{- 2s_0 } = C(s_0) < + \infty$, it follows that
$$
\begin{aligned}
\| A u \|_{H^s}^2 & \lesssim_s \sum_{\eta \in \Z^d} \Big( \sum_{\xi \in \Z^d}  \langle \xi \rangle^s \langle \eta - \xi \rangle^{s_0 + 1} |\widehat a(\eta - \xi)| |\widehat u(\xi)| \frac{1}{\langle \eta - \xi \rangle^{s_0 }} \Big)^2 \\
& \lesssim_s  \sum_{\eta, \xi \in \Z^d} \langle \xi \rangle^{2 s} |\widehat u(\xi)|^2 \langle \eta - \xi \rangle^{2 (s_0 + 1)} |\widehat a(\eta - \xi)|^2 \lesssim_s \| a \|_{H^{s_0 + 1}}^2 \| u \|_{H^s}^2 
\end{aligned}
$$
which gives the claimed bound. 
\end{proof}

\subsection*{Acknowledgements}

\noindent D.M. Ambrose is grateful to the National Science Foundation for support through grant
DMS-2307638.
A. Mazzucato is supported by the US National Science Foundation grants DMS-2206453, DMS-2511023.

\noindent
R. Montalto  is  supported by the ERC STARTING GRANT 2021 “Hamiltonian Dynamics, Normal Forms and Water Waves” (HamDyWWa), Project Number: 101039762.  The Views and opinions expressed are however those of the authors only and do not necessarily reflect those of the European Union or the European Research Council. Neither the European Union nor the granting authority can be held responsible for them.

 \subsection*{Statements and declarations}

\medskip

\noindent
The authors state that there is no conflict of interest and certify that they  have no affiliations or involvement with any organization or entity with any
financial interest, or non-financial interest in the subject matter or materials discussed in this manuscript.

Moreover, data sharing is not applicable to this article as no datasets were generated or analyzed during the
current study.


\begin{thebibliography}{100}

\bibitem{ALN24}
D.M. Ambrose, M.C. Lopes Filho, and H.J. Nussenzveig Lopes. {\it Existence and analyticity of solutions of the Kuramoto–Sivashinsky equation with singular data.} Proc. Roy. Soc. Edinburgh Sect. A. Accepted,  2024.


\bibitem{AM19} D.M. Ambrose and A.L. Mazzucato. {\it Global existence and analyticity for the 2D
Kuramoto-Sivashinsky equation.} J. Dynam. Differential Equations, 31(3):1525–1547, 2019.

\bibitem{AM21} D.M. Ambrose and A.L. Mazzucato. {\it Global solutions of the two-dimensional 
Kuramoto-Sivashinsky equation with a linearly growing mode in each direction.}
J. Nonlinear Sci., 31(6):96, 2021.

\bibitem{BBT03} H. Bellout, S. Benachour and E.S. Titi. {\it Finite-time singularity versus global regularity for hyper-viscous Hamilton-Jacobi-like equations.} Nonlinearity, 16(6):1967-1989, 2003.

\bibitem{BKRZ14} S. Benachour, I. Kukavica, W. Rusin, and M. Ziane. {\it Anisotropic estimates for the two-dimensional Kuramoto-Sivashinsky equation.} J. Dynam. Differential Equations, 26(3):461–
476, 2014.


\bibitem{BS07}   A. Biswas and D. Swanson. {\it Existence and generalized Gevrey regularity of solutions
to the Kuramoto-Sivashinsky equation in $\mathbb{R}^{n}$.} J. Differential Equations, 240(1):145–163, 2007


\bibitem{BG06} J.C. Bronski and T.N. Gambill. {\it Uncertainty estimates and $L^2$ bounds for the
Kuramoto-Sivashinsky equation.} Nonlinearity, 19(9):2023–2039, 2006.

\bibitem{CEES93}  P. Collet, J.-P. Eckmann, H. Epstein, and J. Stubbe. {\it Analyticity for the Kuramoto-Sivashinsky equation.} Phys. D, 67(4):321–326, 1993


\bibitem{CZDFM21} M. Coti-Zelati, M. Dolce, Y Feng, A Mazzucato. {\it Global existence for the two-dimensional Kuramoto–Sivashinsky equation with a shear flow.} Journal of Evolution Equations, 21(4), pp. 5079–5099, 2021.

\bibitem{FM22} Y. Feng, A. Mazzucato. {\it Global existence for the two-dimensional Kuramoto-Sivashinsky equation with advection.} Comm. Partial Differential Equations, 47(2), pp. 279–306, 2022.

\bibitem{FNST88}
C. Foias, B. Nicolaenko, G.R. Sell, and R. Temam. {\it Inertial manifolds for the Kuramoto-Sivashinsky equation and an estimate of their lowest dimension.} J. Math. Pures Appl., 67(3):197-226, 1988.

\bibitem{GMP08} V. A. Galaktionov, È. Mitidieri, and S. I. Pokhozhaev. {\it Existence and nonexistence of global
solutions of the Kuramoto-Sivashinsky equation.} Dokl. Akad. Nauk, 419(4):439–442, 2008. 


\bibitem{GO05} L. Giacomelli and F. Otto. {\it New bounds for the Kuramoto-Sivashinsky equation.} Comm. Pure
Appl. Math., 58(3):297–318, 2005.


\bibitem{GJO15} M. Goldman, M. Josien, and F. Otto. {\it New bounds for the inhomogenous
Burgers and the Kuramoto-Sivashinsky equations.} Comm. Partial Differential Equations,
40(12):2237–2265, 2015.


\bibitem{GF19}   D. Goluskin and G. Fantuzzi. {\it Bounds on mean energy in the Kuramoto-Sivashinsky
equation computed using semidefinite programming.} Nonlinearity, 32(5):1705–1730, 2019.



\bibitem{Good94} J. Goodman. {\it Stability of the Kuramoto-Sivashinsky and related systems.} Comm. Pure Appl.
Math., 47(3):293–306, 1994.

\bibitem{Grenier} E. Grenier. {\it On the nonlinear instability
of Euler and Prandtl equations.} Comm. Pure Appl. Math. 53:1067–1091, 2000.


\bibitem{GN19} E. Grenier, T.T.  Nguyen. {\it 
 Instability of Prandtl Layers.} Ann. PDE 5(2):18, 2019.

\bibitem{Gru00}  Z. Grujić. Spatial analyticity on the global attractor for the Kuramoto-Sivashinsky equation. J. Dynam. Differential Equations, 12(1):217–228, 2000.


 \bibitem{SSArXiv07} X. Ioakim and Y.-S. Smyrlis. {\it Analyticity for Kuramoto-Sivashinsky-type equations in two
spatial dimensions.}  Math. Methods Appl. Sci., 39(8):2159–2178, 2016.

\bibitem{JKT90} M.S. Jolly, I.G. Kevrekidis, and E.S. Titi, 
{\it Approximate inertial manifolds for the Kuramoto-Sivashinsky equation: analysis and computations.} 
Phys. D, 44(1-2):38-60, 1990.

\bibitem{KKP15} A. Kalogirou, E.E. Keaveny, and D.T. Papageorgiou.
{\it An in-depth numerical study of the two-dimensional Kuramoto-Sivashinsky equation.}
Proc. A, 471(2179):20140932, 2015. 

\bibitem{KOS-TITI} A. Kostianko, E. Titi, S. Zelik. {\it Large dispersion, averaging and attractors:
three 1D paradigms.} Nonlinearity,   R317–R350 (2018)

\bibitem{KM23} I. Kukavica and D. Massatt. {\it On the global existence for the Kuramoto-Sivashinsky equation.} J. Dynam. Differential Equations, 35(1):69-85, 2023.

\bibitem{Kura} Y. Kuramoto, {\it Diffusion-Induced Chaos in Reaction Systems.} Progress of Theoretical Physics Supplement, 64 (197802), 346–367.

\bibitem{KT76}  Y. Kuramoto, T. Tsuzuki, {\it Persistent propagation of concentration waves in dissipative media far from thermal equilibrium.} Prog. Theo. Phys. 55:356–369, 1976.

\bibitem{LY20} A. Larios and K. Yamazaki. On the well-posedness of an anisotropically-reduced twodimensional Kuramoto-Sivashinsky equation. Phys. D, 411:132560, 14, 2020.


\bibitem{Metivier} G. Metivier. {\it Para-Differential Calculus and Applications to the Cauchy Problem for Nonlinear Systems.}  Edizioni della Normale, CRM Series (Volume 5), 2008.

\bibitem{Mol00}  L. Molinet. {\it A bounded global absorbing set for the Burgers-Sivashinsky equation in space
dimension two.} C. R. Acad. Sci. Paris Sér. I Math., 330(7):635–640, 2000.

\bibitem{Mol00b} L. Molinet.  {\it Local dissipativity in {$L^2$} for the
              {K}uramoto-{S}ivashinsky equation in spatial dimension 2.}
              J. Dynam. Differential Equations, 12(3):533-556, 2000.



\bibitem{NST85} B. Nicolaenko, B. Scheurer, and R. Temam. {\it Some global dynamical properties of the
Kuramoto-Sivashinsky equations: nonlinear stability and attractors.} Phys. D, 16(2):155–183,
1985.

\bibitem{Otto09}   F. Otto. {\it Optimal bounds on the Kuramoto-Sivashinsky equation.} J. Funct. Anal.,
257(7):2188–2245, 2009.

\bibitem{Pazy83} A. Pazy. Semigroups of linear operators and applications to partial differential equations. {\it Applied Mathematical Sciences}, 44. Springer-Verlag, New York, 1983. viii+279 pp.

\bibitem{SellT92} G.R. Sell and M. Taboada. {\it Local dissipativity and attractors for the Kuramoto-Sivashinsky
equation in thin 2D domains.} Nonlinear Anal., 18(7):671–687, 1992.

\bibitem{Siv77} G.I. Sivashinsky, {\it Nonlinear analysis of hydrodynamic instability in laminar flames—i. derivation of basic equations.} Acta Astronautica 4(11):1177–1206, 1977, no. 11, 

\bibitem{Stan} M. Stanislavova, A. Stefanov. {\it Effective estimates of the higher Sobolev norms for the Kuramoto-Sivashinsky equation.} Conference Publications, 2009, 2009(Special): 729-738. doi: 10.3934/proc.2009.2009.729

\bibitem{Stan2} M. Stanislavova, A. Stefanov. {\it The Kuramoto-Sivashinsky equation in ${\mathbb R}^1$ and ${\mathbb R}^2$: effective estimates of the high-frequency tails and higher Sobolev norms.} arXiv:0711.4005 (2007).

\bibitem{Stan3} M. Stanislavova, A. Stefanov. {\it Asymptotic estimates and stability analysis
of Kuramoto-Sivashinsky type models} J. Evol. Equ. 11, 605–635 (2011). 

\bibitem{Tad86}  E. Tadmor. {\it The well-posedness of the Kuramoto-Sivashinsky equation.} SIAM J. Math. Anal.,
17(4):884–893, 1986.

\bibitem{Taylor} M. Taylor. {\it  Pseudodifferential Operators and Nonlinear PDEs.} Birkhauser, Boston, 1991

\bibitem{TKP18}
R.J. Tomlin, A. Kalogirou, and D.T. Papageorgiou. {\it Nonlinear dynamics of a dispersive anisotropic Kuramoto-Sivashinsky equation in two space dimensions.} Proc. A, 474(2211):20170687, 2018.

\end{thebibliography}
\end{document}